\documentclass[a4paper,12pt]{article}

\usepackage{amsmath}
\usepackage{amssymb}
\usepackage{latexsym}
\usepackage{graphicx}
\usepackage{psfrag}
\usepackage{hyperref}

\newtheorem{theorem}{Theorem}[section]
\newtheorem{lemma}{Lemma}[section]
\newtheorem{corollary}{Corollary}[section]
\newtheorem{prop}{Proposition}[section]
\newtheorem{definition}{Definition}[section]
\newtheorem{remark}{Remark}
\newtheorem{example}{Example}

\newenvironment{proof}{\noindent{\textsc{Proof.}}}
{$\hfill\Box$\vspace{0.1 cm}\\}

\newcommand{\R}{\mathbb R}

\newcommand{\fhi}{\varphi}

\newcommand{\abs}[1]{\left\vert #1 \right\vert}

\newcommand{\KK}{\mathcal{K}}

\newcommand{\Rr}{\mathcal{RS}}
\newcommand{\Rsol}{\mathcal{RS}}
\newcommand{\s}{\subseteq}

\newcommand{\entropia}{the entropy condition (E1) }
\newcommand{\f}[1]{f(\bar\rho_{#1})}
\newcommand{\ff}[1]{f(#1)}
\newcommand{\brho}{\bar\rho}
\newcommand{\OmegaH}{\Omega_{\mathcal H}}
\newcommand{\PhiH}{\Phi_{\mathcal H}}
\newcommand{\entropy}[1]{\mathcal F\left( #1 \right) \ge 0}
\newcommand{\srho}{\brho_1,\brho_2,\brho_3,\brho_4}

\DeclareMathOperator{\sgn}{sgn}

\begin{document}

%\date{}

\title{Entropy type conditions for Riemann solvers at nodes}

\author{Mauro Garavello\thanks{E-mail: \texttt{mauro.garavello@mfn.unipmn.it}.
    Partially supported by Dipartimento di Matematica e Applicazioni,
    Universit\`a di Milano-Bicocca.
  }\\
  Dipartimento di Scienze e Tecnologie Avanzate,\\
  Universit\`a del Piemonte Orientale ``A. Avogadro'',\\
  viale Teresa Michel 11,
  15121 Alessandria (Italy). \\
  \and 
  Benedetto Piccoli\thanks{E-mail: \texttt{piccoli@iac.rm.cnr.it}.
  }\\
  I.A.C., C.N.R.,\\
  Via dei Taurini 19, 00185 Roma
  (Italy).
}

\maketitle

\begin{abstract}
  This paper deals with conservation laws on networks, represented by graphs.
  Entropy-type conditions are considered to determine dynamics at nodes.
  Since entropy dispersion is a local concept, we consider a network
  composed by a single node
  $J$ with $n$ incoming and $m$ outgoing arcs.
  We extend at $J$ the classical Kru\v{z}kov entropy
  obtaining two conditions, denoted by (E1) and (E2): the first
  requiring entropy
  condition for all Kru\v{z}kov entropies, the second only for the value
  corresponding to sonic point. First we show that in case $n \ne m$,
  no Riemann solver can satisfy the strongest condition.
  Then we characterize all the Riemann solvers at $J$ satisfying
  the strongest condition (E1), in the case of nodes
  with at most two incoming and two outgoing arcs.
  Finally we focus three different Riemann solvers, introduced in
  previous papers.
  In particular, we show that the Riemann solver introduced for data networks
  is the only one always satisfying (E2).
\end{abstract}



%
%
% introduction
%
%
\section{Introduction}\label{se:introduction}

Nonlinear hyperbolic conservation laws on networks have recently attracted
a lot of interest in various fields: car traffic
\cite{C-G-P,garavello-piccoli_ARModel_2006,gp-book,Holden-Risebro_siam_1995},
gas dynamics~\cite{MR2418671,MR2223073,MR2219276,MR2247787,MR2311526,
  MR2377285,colombo-guerra-herty-sachers,
  MR2438778,ColomboMarcellini,MR2441091},
irrigation channels~\cite{MR2357767,MR2164806,MR2055319,MR1920161}
and supply chains~\cite{MR2357763,MR2318380}. 
A network is modeled by a graph: a finite collection of arcs
connected together by vertices.
On each arc we consider a scalar conservation law.
For instance one may think to
the Lighthill-Whitham-Richards model for car
traffic~\cite{Lighthill-Whitham_1955,Richards_1956}. However, our
results applies to the other application domains.

It is easy to check that the dynamic at nodes is not uniquely determined
by imposing the conservation of mass through vertices.
Then, to completely describe the network load evolution,
the first step is to appropriately define the concept of solution at a
vertex.\\
As in the classical theory of conservation laws, this problem
is equivalent to giving the solution Riemann problems (now at vertices).
More precisely, a Riemann problem at a vertex is simply a Cauchy
problem with constant
initial conditions in each arc of the vertex.
The map, which associates the solution to each Riemann problem at a vertex $J$,
is called a Riemann solver at $J$.
Similarly to the case of a real line, one has to resort to the concept
of weak solutions
in the sense of distributions and there are infinitely many Riemann solvers
producing weak solutions.
First one uses entropy type conditions inside arcs as for the real line. Then,
in order to select a particular solution (i.e. a Riemann solver)
at the vertex, one has to impose some additional conditions.
In~\cite{C-G-P}, for example, the authors required some rules about the
distribution of the fluxes in the arcs and a maximization
condition; see also~\cite{da-m-p,marigo-piccoli_2008_T_junction}.
It is then natural to ask if entropy-like conditions can be imposed also
at the vertex and not only inside arcs.

In this paper, we focus on a single vertex $J$, composed by $n$ incoming
and $m$ outgoing arcs and we extend the Kru\v{z}kov~\cite{MR0267257}
entropy-type conditions.
More precisely, we propose two different entropy conditions
for admissibility of solutions, called, respectively, (E1) and (E2).
The condition~(E1) is stronger than~(E2), indeed the first asks for Kru\v{z}kov
entropy condition to be verified for all entropies, while the second
asks only for the precise Kru\v{z}kov entropy corresponding to sonic point.
It is interesting to note that the entropy condition~(E1)
imposes strong restrictions both on Riemann solvers and on the geometry of
the vertex. Indeed, Riemann solvers satisfying~(E1) can exist only in the
case of vertices with the same number of incoming and outgoing arcs.

We then test our conditions on Riemann solvers considered in the literature.
First we can prove that the Riemann solver, introduced in~\cite{da-m-p}
for data networks, satisfies~(E2) and, in special situations, also~(E1).\\
Then we show that the Riemann solvers defined
in~\cite{C-G-P,marigo-piccoli_2008_T_junction} do not satisfy~(E2).
However, at least for the Riemann solver in~\cite{C-G-P}, the entropy
condition and the maximization procedure agree on some particular set,
over which the maximization is taken. Roughly speaking the solver
respects the entropy condition once traffic distribution is imposed.

The paper is organized as follows. Section~\ref{se:def} introduces the
basic definitions of networks and of solutions.
Section~\ref{se:Riemann_problem} deals with the solution to the Riemann
problem at the vertex $J$. Moreover, we introduce the entropy
conditions (E1) and (E2) for Riemann solvers at $J$.
In Section~\ref{se:RS_E1}, we determine which Riemann solvers satisfy
the entropy condition (E1). The paper ends with Section~\ref{se:examples},
which considers the Riemann solvers $\Rr_1$, $\Rr_2$ and $\Rr_3$,
introduced respectively in \cite{C-G-P,da-m-p,marigo-piccoli_2008_T_junction},
and analyzes what entropy conditions these Riemann solvers satisfy.

%
%
% definitions
%
%
\section{Basic Definitions and Notations}\label{se:def}

Consider a node $J$ with $n$ incoming arcs $I_1,\ldots,I_n$ and $m$
outgoing arcs $I_{n+1},\ldots,I_{n+m}$.
We model each incoming arc $I_i$ ($i\in\{1,\ldots,n\}$)
of the node with the real interval $I_i=]-\infty,0]$ and
each outgoing arc $I_j$ ($j\in\{n+1,\ldots,n+m\}$)
of the node with the real interval $I_j=[0,+\infty[$.
On each arc $I_l$ ($l\in\{1,\ldots,n+m\}$), the traffic evolution
is given by
\begin{equation}
  \label{eq:LWR}
  (\rho_l)_t+f(\rho_l)_x=0,
\end{equation}
where $\rho_l=\rho_l(t,x)\in [0,\rho_{max}]$,
is the {\em density}, $v_l=v_l(\rho_l)$
is the {\em average velocity}
and $f(\rho_l)=v_l(\rho_l)\,\rho_l$ is the {\em flux}.
Hence the network load is described by a finite collection of
functions $\rho_l$ defined on $[0,+\infty[\times I_l$.
For simplicity, we put $\rho_{max}=1$.
On the flux $f$ we make the following assumption
\begin{itemize}
\item[{\bf (${\cal F}$)}]
  $f : [0,1] \rightarrow \R$ is a piecewise smooth concave function satisfying
  \begin{enumerate}
  \item $f(0)=f(1)=0$;

  \item there exists a unique $\sigma\in]0,1[$ such that $f$ is strictly
    increasing in $[0,\sigma[$ and strictly decreasing in $]\sigma,1]$.
  \end{enumerate}
\end{itemize}

\begin{definition} \label{deftau}
  Let $\tau:[0,1] \rightarrow [0,1]$ be the map such that:
  \begin{enumerate}
  \item $f(\tau(\rho))=f(\rho)$ for every $\rho\in[0,1]$;
    
  \item $\tau(\rho) \not= \rho$ for every $\rho\in[0,1]\setminus\{\sigma\}$.
  \end{enumerate}
\end{definition}

\begin{definition}
  A function $\rho_l\in C([0,+\infty[;L^1_{loc}(I_l))$ is an entropy-admissible
  solution to~(\ref{eq:LWR}) in the arc $I_l$ if,
  for every $k \in [0,\rho_{max}]$ and every
  $\tilde\varphi:[0,+\infty[\times I_l\to\R$ 
  smooth, positive with compact support in
  $]0,+\infty[\times \left(I_l\setminus\{0\}\right)$
  \begin{equation} \label{eq:entsol-oneroad}
    \int_0^{+\infty}\int_{I_l}\Big( |\rho_l -k|{\frac{\partial \tilde\varphi}
      {\partial t}} + \sgn(\rho_l-k)(f(\rho_l)- f(k))
    {\frac{\partial \tilde\varphi}
      {\partial x}}\Big)dxdt \geq 0.
  \end{equation}
\end{definition}

\begin{definition}\label{def:weak_solution}
  A collection of functions $\rho_l\in C([0,+\infty[;L^1_{loc}(I_l))$,
  ($l\in\{1,\ldots,n+m\}$) is a weak solution at $J$ if
  \begin{enumerate}
  \item for every $\l\in\{1,\ldots,n+m\}$, the function
    $\rho_l$ is an entropy-admissible solution to~(\ref{eq:LWR})
    in the arc $I_l$;

  \item for every $\l\in\{1,\ldots,n+m\}$ and for a.e. $t>0$, the function
    $x\mapsto\rho_l(t,x)$ has a version with bounded total variation;

  \item for a.e. $t>0$, it holds
    \begin{equation}\label{eq:RH}
      \sum\limits_{i=1}^n f(\rho_i (t, 0-)) = \sum\limits_{j=n
        +1}^{n+m}f(\rho_j (t, 0+))\,,
    \end{equation}
    where $\rho_l$ stands for the version with bounded total variation.
  \end{enumerate}
\end{definition}

%
%
% Riemann problem
%
%
\section{The Riemann Problem at $J$}\label{se:Riemann_problem}

Given $\rho_{1,0},\ldots,\rho_{n+m,0}\in[0,1]$, a Riemann problem at $J$
is a Cauchy problem at $J$ with constant initial data on each arc, i.e.
\begin{equation}
  \label{eq:RPatJ}
  \left\{
    \begin{array}{ll}
      \begin{array}{l}
        \frac\partial{\partial t}\rho_l+\frac\partial{\partial x}f(\rho_l)=0,
        \vspace{.2cm}\\
        \rho_l(0,\cdot)=\rho_{0,l}, 
      \end{array}
      & l\in\{1,\ldots,n+m\}.
    \end{array}
  \right.
\end{equation}

Now, we give some definitions for later use. The first one is the definition
of Riemann solver, which is a map giving a solution to the Riemann
problem~\eqref{eq:RPatJ}.

\begin{definition}\label{def:Riemann_solver}
  A Riemann solver $\mathcal{RS}$ is a function
  \begin{equation*}
    \begin{array}{rccc}
      \mathcal{RS}: & [0,1]^{n+m} & \longrightarrow
      & [0,1]^{n+m}\\
      & (\rho_{1,0},\ldots,\rho_{n+m,0}) & \longmapsto & 
      (\bar\rho_1,\ldots,\bar\rho_{n+m})
    \end{array} 
  \end{equation*}
  satisfying
  \begin{enumerate}
  \item \label{enum:1_def_RS}
    $\sum_{i=1}^nf(\bar\rho_i)=\sum_{j=n+1}^{n+m}f(\bar\rho_j)$;

  \item \label{enum:2_def_RS}
    for every $i\in\{1,\ldots,n\}$, the classical Riemann problem
    \begin{equation*}
      \left\{
        \begin{array}{l}
          \rho_t+f(\rho)_x=0,\hspace{1cm}x\in\R,\, t>0,\vspace{.2cm}\\
          \rho(0,x)=\left\{
            \begin{array}{ll}
              \rho_{i,0}, & \textrm{ if } x<0,\\
              \bar\rho_i, & \textrm{ if } x>0,
            \end{array}
          \right.
        \end{array}
      \right.
    \end{equation*}
    is solved with waves with negative speed;

  \item \label{enum:3_def_RS}
    for every $j\in\{n+1,\ldots,n+m\}$, the classical Riemann problem
    \begin{equation*}
      \left\{
        \begin{array}{l}
          \rho_t+f(\rho)_x=0,\hspace{1cm}x\in\R,\, t>0,\vspace{.2cm}\\
          \rho(0,x)=\left\{
            \begin{array}{ll}
              \bar\rho_j, & \textrm{ if } x<0,\\
              \rho_{j,0}, & \textrm{ if } x>0,
            \end{array}
          \right.
        \end{array}
      \right.
    \end{equation*}
    is solved with waves with positive speed.
  \end{enumerate}
\end{definition}

We introduce the concepts of equilibrium and consistency for Riemann solvers.
The fixed points of a Riemann solver are called equilibria,
while a Riemann solver has the consistency
condition when its image is contained in the equilibria.

%
% property equilibrium
%
\begin{definition}\label{def:equilibrium}
  We say that $(\rho_{1,0},\ldots,\rho_{n+m,0})$ is an equilibrium for the
  Riemann solver $\Rsol$ if
  \begin{equation*}
    \Rsol(\rho_{1,0},\ldots,\rho_{n+m,0})=(\rho_{1,0},\ldots,\rho_{n+m,0}).
  \end{equation*}
\end{definition}

%
% consistency
%
\begin{definition}\label{def:consistency}
  We say that a Riemann solver $\Rsol$ satisfies the consistency condition if,
  for every $(\rho_{1,0},\ldots,\rho_{n+m,0})\in[0,1]^{n+m}$, then
  $\Rsol(\rho_{1,0},\ldots,\rho_{n+m,0})$ is an equilibrium for $\Rsol$.
\end{definition}

We introduce now the concepts of entropy functions and admissible entropy
conditions (E1) and (E2) for Riemann solvers. We are essentially extending the
Kru\v zkov entropy condition to the case of a node; see~\cite{MR0267257}.

%
% entropy function
%
\begin{definition}
  The function $\mathcal F:[0,1]^{n+m}\times[0,1]\to\R$, defined by
  \begin{eqnarray}
    \label{eq:entropy_flux_function}
    \mathcal F(\rho_1,\ldots,\rho_{n+m},k) & = &
    \sum_{i=1}^n\sgn(\rho_i- k)\left(f(\rho_i)-f(k)\right)\\
    & & -\sum_{j=n+1}^{n+m}\sgn(\rho_j- k)
    \left(f(\rho_j)-f(k)\right),\nonumber
  \end{eqnarray}
  is called entropy-flux function.
\end{definition}

\begin{definition}\label{def:entropy_RS_E1}
  A Riemann solver $\Rsol$ satisfies the entropy condition (E1) if,
  for every initial condition $(\rho_{1,0},\ldots,\rho_{n+m,0})$
  and for every $k\in[0,1]$, we have
  \begin{equation}
    \label{eq:entropy_RS_E1}
    \mathcal F(\bar\rho_1,\ldots,\bar\rho_{n+m},k)\ge0,
  \end{equation}
  where $(\bar\rho_1,\ldots,\bar\rho_{n+m})
  =\Rsol(\rho_{1,0},\ldots,\rho_{n+m,0})$.
\end{definition}

\begin{remark}
  If $k=0$, then equation~(\ref{eq:entropy_RS_E1}) becomes
  $\sum_{i=1}^nf(\bar\rho_i)\ge\sum_{j=n+1}^{n+m}\f j$.\\
  If $k=1$, then equation~(\ref{eq:entropy_RS_E1}) becomes
  $\sum_{i=1}^nf(\bar\rho_i)\le\sum_{j=n+1}^{n+m}\f j$. Therefore the entropy
  condition (E1) implies the conservation identity
  $\sum_{i=1}^nf(\bar\rho_i)=\sum_{j=n+1}^{n+m}\f j$.
\end{remark}

\begin{definition}\label{def:entropy_RS_E2}
  A Riemann solver $\Rsol$ satisfies the entropy condition (E2) if,
  for every initial condition $(\rho_{1,0},\ldots,\rho_{n+m,0})$, we have
  \begin{equation}
    \label{eq:entropy_RS_E2}
    \mathcal F(\bar\rho_1,\ldots,\bar\rho_{n+m},\sigma)\ge0,
  \end{equation}
  where $(\bar\rho_1,\ldots,\bar\rho_{n+m})
  =\Rsol(\rho_{1,0},\ldots,\rho_{n+m,0})$.
\end{definition}

\begin{remark}
  The entropy condition~(\ref{eq:entropy_RS_E1}) can be deduced in the
  following way.

  Fix, for every $l \in \{ 1, \ldots, n+m \}$, a smooth function
  $\fhi_l: [0, +\infty [ \times I_l \to [0, +\infty[$
  with support contained in $[0, +\infty [ \times [-M,M]$
  for some $M > 0$ and assume that
  $\fhi_{l'}(t,0)=\fhi_{l''}(t,0)$ for every $t \ge 0$ and
  $l', l'' \in \{ 1, \ldots, n+m \}$.
  Applying the divergence theorem to the inequality
  \begin{equation*}
    \sum_{l=1}^{n+m} \int_0^{+\infty} \int_{I_l}
    \left[ \abs{\bar \rho_l- k} \fhi_{l,t} +
      \sgn(\bar \rho_l- k) \left( f(\bar \rho_l) - f( k) \right) \fhi_{l,x}
    \right] dx dt \ge 0,
  \end{equation*}
  where $(\bar \rho_1, \ldots, \bar \rho_{n+m})$ is an equilibrium at $J$,
  we deduce~\eqref{eq:entropy_RS_E1}.

  Obviously, these kinds of entropies are not justified
  by physical considerations.  
\end{remark}

Finally, let us introduce sets $\Omega_l$ and $\Phi_l$,
related to the points~\ref{enum:2_def_RS} and~\ref{enum:3_def_RS}
of Definition~\ref{def:Riemann_solver}.

\begin{enumerate}
\item For every $i\in\{1,\ldots,n\}$ define
  \begin{equation}
    \label{eq:omega_i}
    \Omega_i=\left\{
      \begin{array}{ll}
        [0,f(\rho_{i,0})], & \textrm{ if } 0\le\rho_{i,0}\le\sigma,
        \vspace{.2cm}\\
        { }[0,f(\sigma)], & \textrm{ if } \sigma\le\rho_{i,0}\le1,
      \end{array}
    \right.
  \end{equation}
  and
  \begin{equation}
    \label{eq:Phi_i}
    \Phi_i=\left\{
      \begin{array}{ll}
        \{\rho_{i,0}\}\cup]\tau(\rho_{i,0}),1], &
        \textrm{ if } 0\le\rho_{i,0}\le\sigma,\vspace{.2cm}\\
        { }[\sigma,1], & \textrm{ if } \sigma \le \rho_{i,0} \le 1.
      \end{array}
    \right.
  \end{equation}

\item For every $j\in\{n+1,\ldots,n+m\}$ define
  \begin{equation}
    \label{eq:omega_j}
    \Omega_j=\left\{
      \begin{array}{ll}
        [0,f(\sigma)], & \textrm{ if } 0\le\rho_{j,0}\le\sigma,
        \vspace{.2cm}\\
        { }[0,f(\rho_{j,0})], & \textrm{ if } \sigma\le\rho_{j,0}\le1,
      \end{array}
    \right.
  \end{equation}
  and
  \begin{equation}
    \label{eq:Phi_j}
    \Phi_j=\left\{
      \begin{array}{ll}
        { }[0,\sigma], & \textrm{ if } 0\le\rho_{j,0}\le\sigma,\vspace{.2cm}\\
        \{\rho_{j,0}\}\cup[0,\tau(\rho_{j,0})[, &
        \textrm{ if } \sigma\le\rho_{j,0}\le1.
      \end{array}
    \right.
  \end{equation}
\end{enumerate}

The following Proposition links the previous sets
with Definition~\ref{def:Riemann_solver}.

\begin{prop}\label{prop:stati_ammissibili}
  The following statements hold.
  \begin{enumerate}
  \item For every $i\in\{1,\ldots,n\}$, an element $\bar\gamma$ belongs to
    $\Omega_i$ if and only if there exists $\bar\rho_i\in[0,1]$ such that
    $f(\bar\rho_i)=\bar\gamma$ and point~\ref{enum:2_def_RS} of
    Definition~\ref{def:Riemann_solver} is satisfied.

  \item For every $j\in\{n+1,\ldots,n+m\}$, an element $\bar\gamma$ belongs to
    $\Omega_j$ if and only if there exists $\bar\rho_j\in[0,1]$ such that
    $f(\bar\rho_j)=\bar\gamma$ and point~\ref{enum:3_def_RS} of
    Definition~\ref{def:Riemann_solver} is satisfied.
  \end{enumerate}
\end{prop}

The proof is trivial and hence omitted.
The main result of this Section is that, if $n \ne m$, then
every Riemann solver $\Rsol$ at $J$ does not satisfy the entropy
condition (E1). We first need the following result.

\begin{prop}\label{prop:n_ne_m}
  Fix a node $J$ with $n$ incoming arcs and $m$ outgoing arcs and
  a Riemann solver $\Rr$ satisfying the entropy condition (E1).
  Denote with $(\bar\rho_1,\ldots,\bar\rho_{n+m})$ the image through
  $\Rr$ of the initial condition $(\rho_{1,0},\ldots,\rho_{n+m,0})$.
  \begin{enumerate}
  \item If $n>m$, then $\min\left\{\bar\rho_1,\ldots,\bar\rho_{n}\right\}=0$.

  \item If $n<m$, then
    $\max\left\{\bar\rho_{n+1},\ldots,\bar\rho_{n+m}\right\}=1$.
  \end{enumerate}
\end{prop}

\begin{proof}
  Consider first the case $n>m$. Suppose by contradiction that
  $\min\left\{\bar\rho_1,\ldots,\bar\rho_{n}\right\}>0$.
  Define the set $J=\left\{j\in\{n+1,\ldots,n+m\}\,:\bar\rho_j=0\right\}$ and
  fix $0<k<\min\left\{\bar\rho_l\,:l\in\{1,\ldots,n+m\}\setminus J\right\}$.
  Thus, the entropy inequality
  $\entropy{\bar \rho_1, \ldots, \bar \rho_{n+m}, k}$ becomes,
  \begin{displaymath}
    \sum_{i=1}^n\left[f(\bar\rho_i)-f(k)\right]\ge
    \sum_{j\in\{n+1,\ldots,n+m\}\setminus J}\left[f(\bar\rho_j)-f(k)\right]
    +\sum_{j\in J}f(k).
  \end{displaymath}
  By point~\ref{enum:1_def_RS} of
  Definition~\ref{def:Riemann_solver}, we deduce that
  \begin{displaymath}
    -nf(k)\ge -(m-\#(J))f(k)+\#(J)f(k),
  \end{displaymath}
  where $\# (J)$ denotes the cardinality of $J$; thus
  $(m-n-2\#(J))f(k)\ge0$, which is a contradiction.

  Consider now the situation $n<m$. By contradiction we assume that
  $\max\left\{\bar\rho_{n+1},\ldots,\bar\rho_{n+m}\right\}<1$.
  Define the set $I=\left\{i\in\{1,\ldots,n\}\,:\bar\rho_i=1\right\}$ and
  fix $\max\left\{\bar\rho_l\,:l\in\{1,\ldots,n+m\}\setminus I\right\}<k<1$.
  Thus, the entropy inequality
  $\entropy{\bar \rho_1, \ldots, \bar \rho_{n+m}, k}$ becomes,
  \begin{displaymath}
    \sum_{i\in\{1,\ldots,n\}\setminus I}\left[f(k)-f(\bar\rho_i)\right]
    -\sum_{i\in I}f(k)\ge
    \sum_{j=n+1}^{n+m}\left[f(k)-f(\bar\rho_j)\right].
  \end{displaymath}
  By point~\ref{enum:1_def_RS}
  of Definition~\ref{def:Riemann_solver}, we deduce that
  $(n-2\#(I)-m)f(k)\ge0$, which is a contradiction.
\end{proof}

\begin{theorem}
  \label{thm:n_ne_m}
  Fix a node $J$ with $n$ incoming arcs and $m$ outgoing arcs
  and suppose that $n\ne m$.
  Every Riemann solver $\Rsol$ at $J$ does
  not satisfy the entropy condition (E1).
\end{theorem}

\begin{proof}
  Suppose, by contradiction, that there exists a
  Riemann solver $\Rr$ at $J$ satisfying the entropy condition (E1).

  Assume $n>m$ and consider an initial condition
  $(\rho_{1,0},\ldots,\rho_{n+m,0})$ satisfying $\rho_{i,0}\ne0$ for
  every $i\in\{1,\ldots,n\}$. If
  $(\bar\rho_1,\ldots,\bar\rho_{n+m})=\Rr(\rho_{1,0},\ldots,\rho_{n+m,0})$,
  then, by Proposition~\ref{prop:n_ne_m}, 
  there exists $i_1\in\{1,\ldots,n\}$ such that $\bar\rho_{i_1}=0$,
  which is a contradiction since the wave $(\rho_{i_1,0},\bar\rho_{i_1})$
  has not negative speed.

  Assume now $n<m$ and consider an initial condition
  $(\rho_{1,0},\ldots,\rho_{n+m,0})$ satisfying $\rho_{j,0}\ne1$ for
  every $j\in\{n+1,\ldots,n+m\}$.
  By Proposition~\ref{prop:n_ne_m}, if
  $(\bar\rho_1,\ldots,\bar\rho_{n+m})=\Rr(\rho_{1,0},\ldots,\rho_{n+m,0})$,
  then there exists $j_1\in\{n+1,\ldots,n+m\}$ such that $\bar\rho_{j_1}=1$,
  which is a contradiction since the wave $(\bar\rho_{j_1},\rho_{j_1,0})$
  has not positive speed.
\end{proof}

%
%
% section RS satisfying E1
%
%
\section{Riemann solvers satisfying (E1)}\label{se:RS_E1}

In this Section we determine which Riemann solver satisfies the entropy
condition (E1), in the sense of Definition~\ref{def:entropy_RS_E1},
for nodes with $n = m \in \{ 1, 2 \}$.
In the case $n \ne m$, Theorem~\ref{thm:n_ne_m} implies that every Riemann
solver does not satisfy (E1). 
Moreover if $n = m = 1$, then there exists exactly one Riemann solver at $J$
satisfying (E1), while if $n = m = 2$, then there exist infinitely many
Riemann solvers satisfying (E1); see Sections~\ref{sse:n=m=1}
and~\ref{sse:n=m=2}.
We do not treat the case $n = m > 2$, for the huge number of different
situations.

%
% one-one node
%
\subsection{Nodes with $n = m =1$}
\label{sse:n=m=1}

In this subsection, we fix a node $J$ with one incoming and one outgoing
arc. The following result holds.

\begin{prop}\label{prop:1x1}
  A Riemann solver $\Rr$ at $J$ satisfies the entropy condition (E1)
  if and only if, for every initial datum $(\rho_{1,0},\rho_{2,0})$,
  the image $(\bar\rho_1,\bar\rho_2)=\Rr(\rho_{1,0},\rho_{2,0})$
  satisfies either
  \begin{equation}\label{eq:cond1_E1}
    \bar\rho_1=\bar\rho_2
  \end{equation}
  or
  \begin{equation}\label{eq:cond2_E1}
    \bar\rho_1<\bar\rho_2\quad\textrm{ and }\quad f(\bar\rho_1)=f(\bar\rho_2).
  \end{equation} 
\end{prop}

\begin{proof}
  Consider first a Riemann solver $\Rr$ satisfying the entropy condition (E1).
  By~\ref{enum:1_def_RS}
  of Definition~\ref{def:Riemann_solver}, it is clear that
  $f(\bar\rho_1)=f(\bar\rho_2)$. Assume by contradiction that
  $\bar\rho_1>\bar\rho_2$. Since $f(\bar\rho_1)=f(\bar\rho_2)$, we
  easily deduce that $\bar\rho_2<\sigma<\bar\rho_1$.
  Putting $k=\sigma$ in equation~(\ref{eq:entropy_RS_E1}) we derive
  \begin{displaymath}
    f(\bar\rho_1)-f(\sigma)\ge f(\sigma)-f(\bar\rho_2),
  \end{displaymath}
  which is, by assumptions, equivalent to $f(\bar\rho_1)\ge f(\sigma)$,
  and so we get a contradiction.

  Consider now a Riemann solver $\Rr$ such that,
  for every initial datum $(\rho_{1,0},\rho_{2,0})$,
  the image $(\bar\rho_1,\bar\rho_2)=\Rr(\rho_{1,0},\rho_{2,0})$
  satisfies either~(\ref{eq:cond1_E1}) or~(\ref{eq:cond2_E1}).
  It is trivial to prove that (E1) holds.
\end{proof}

\begin{theorem}
  There exists a unique Riemann solver $\Rr$ at $J$ satisfying the entropy
  condition (E1). This Riemann solver satisfies the consistency condition
  and coincides with the Riemann solver introduced in~\cite{C-G-P}
  for traffic or with the Riemann solver introduced in~\cite{da-m-p}.
\end{theorem}

\begin{proof}
  Fix an initial datum $(\rho_{1,0},\rho_{2,0})$. We show that there exists
  a unique $(\bar\rho_1,\bar\rho_2)$, which is the image of an entropy
  admissible Riemann solver.

  If $\rho_{1,0}=\rho_{2,0}$, then we claim that
  $\bar\rho_1=\bar\rho_2=\rho_{1,0}$. Assume by contradiction that
  $\bar\rho_1\ne\bar\rho_2$. In this case either $\bar\rho_1<\sigma<\bar\rho_2$
  or $\bar\rho_2<\sigma<\bar\rho_1$. By Proposition~\ref{prop:1x1}, the
  only possibility is $\bar\rho_1<\sigma<\bar\rho_2$.
  By Proposition~\ref{prop:stati_ammissibili}, either $\bar\rho_1=\rho_{1,0}$
  or $\bar\rho_2=\rho_{2,0}$. In the first case $\bar\rho_2=\tau(\rho_{2,0})$,
  while in the second one $\bar\rho_1=\tau(\rho_{1,0})$. It is not possible.

  Assume now that $\rho_{1,0}\ne\rho_{2,0}$.
  We have some different possibilities.
  \begin{enumerate}
  \item $\max\{\rho_{1,0},\rho_{2,0}\}\le\sigma$.
    By Proposition~\ref{prop:stati_ammissibili}, we deduce that
    $\bar\rho_2\in[0,\sigma]$. Moreover, by Proposition~\ref{prop:1x1},
    we deduce that $\bar\rho_1=\rho_{1,0}$;
    hence $\bar\rho_2=\bar\rho_1=\rho_{1,0}$. This solution respects all the
    properties of Definition~\ref{def:Riemann_solver} and the entropy
    condition~(\ref{eq:entropy_RS_E1}).

  \item $\min\{\rho_{1,0},\rho_{2,0}\}\ge\sigma$.
    By Proposition~\ref{prop:stati_ammissibili}, we deduce that
    $\bar\rho_1\in[\sigma,1]$. Moreover, by Proposition~\ref{prop:1x1},
    we deduce that $\bar\rho_2=\rho_{2,0}$;
    hence $\bar\rho_2=\bar\rho_1=\rho_{2,0}$. This solution respects all the
    properties of Definition~\ref{def:Riemann_solver} and the entropy
    condition~(\ref{eq:entropy_RS_E1}).

  \item $\rho_{1,0}<\sigma<\rho_{2,0}$.
    By Proposition~\ref{prop:stati_ammissibili}, we deduce that
    $\bar\rho_1=\rho_{1,0}$ or $\bar\rho_1>\sigma$ and that
    $\bar\rho_2=\rho_{2,0}$ or $\bar\rho_2<\sigma$.\\
    If $f(\rho_{1,0})=f(\rho_{2,0})$, then,
    by Proposition~\ref{prop:1x1}, the only possibility is that
    $\bar\rho_1=\rho_{1,0}$ and $\bar\rho_2=\rho_{2,0}$.\\
    If $f(\rho_{1,0})>f(\rho_{2,0})$, then,
    by Proposition~\ref{prop:1x1}, the only possibility is that
    $\bar\rho_1=\bar\rho_2=\rho_{2,0}$.\\
    Finally, if $f(\rho_{1,0})<f(\rho_{2,0})$, then,
    by Proposition~\ref{prop:1x1}, the only possibility is that
    $\bar\rho_1=\bar\rho_2=\rho_{1,0}$.\\
    In all the cases, the solution respects all the
    properties of Definition~\ref{def:Riemann_solver} and the entropy
    condition~(\ref{eq:entropy_RS_E1}).

  \item $\rho_{2,0}<\sigma<\rho_{1,0}$.
    By Proposition~\ref{prop:stati_ammissibili}, we deduce that
    $\bar\rho_1\ge\sigma$ and $\bar\rho_2\le\sigma$.
    By Proposition~\ref{prop:1x1}, the only possibility is that
    $\bar\rho_1=\bar\rho_2=\sigma$.
    The solution respects all the
    properties of Definition~\ref{def:Riemann_solver} and the entropy
    condition~(\ref{eq:entropy_RS_E1}).
  \end{enumerate}

  The proof is completed.
\end{proof}

\begin{remark}
  In~\cite{g-n-p-t}, the authors described all the Riemann solvers, with
  suitable properties, for nodes $J$ with $n = m = 1$.
  The unique Riemann solver $\Rsol$ satisfying (E1) corresponds to the Riemann
  solver generated by the set $X = \{ f(\sigma) \}$ and described in
  Section~3.1 of~\cite{g-n-p-t}.
\end{remark}

\begin{remark}
  \label{rmk:f1f2}
  One can try to generalize the entropy condition (E1),
  at least for nodes with $n = m = 1$, to the case of fluxes depending
  on the arcs. Unfortunately this is not a trivial problem.
  Consider indeed the following example. Let $f_1: [0,1] \to \R$,
  $f_2: [0,1] \to \R$ be two fluxes satisfying $(\mathcal F)$ and assume
  that:
  \begin{enumerate}
  \item $f_1$ is the flux in the arc $I_1$;

  \item $f_2$ is the flux in the arc $I_2$;

  \item $\sigma = \frac12$ is the point of maximum for both $f_1$ and $f_2$;

  \item $f_1(\rho) < f_2(\rho)$ for every $\rho \in ]0,1[$.
  \end{enumerate}
  Choose $0 < \bar \rho_2 < \bar \rho_1 < \frac12$ such that
  $f_1(\bar \rho_1) = f_2(\bar \rho_2)$ and take
  $k \in [\bar \rho_2, \bar \rho_1]$; see Figure~\ref{fig:remark_f1f2}.
  Then, the entropy condition~\eqref{eq:entropy_RS_E1} becomes
  \begin{displaymath}
    f_1 (\bar \rho_1) - f_1(k) \ge 
    f_2(k) - f_2(\bar \rho_2),
  \end{displaymath}
  which is equivalent to $f_1(k) + f_2(k) \le f_1 (\bar \rho_1) +
  f_2 (\bar \rho_2)$. The last inequality does not hold for $k = \bar \rho_1$
  and for all $k \in [\bar \rho_2, \bar \rho_1]$ near to $\bar \rho_1$.

  \begin{figure}
    \centering
    \begin{psfrags}
      \psfrag{0}{$0$} \psfrag{1}{$1$} \psfrag{s}{$\frac12$}
      \psfrag{f1}{$f_1$} \psfrag{f2}{$f_2$}
      \psfrag{r1}{$\bar \rho_1$} \psfrag{r2}{$\bar \rho_2$}
      \psfrag{rho}{$\rho$}
      \includegraphics[width=6cm]{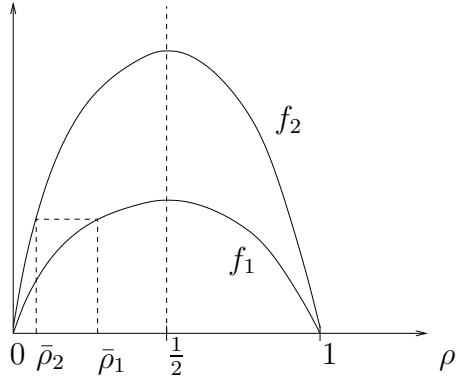}
    \end{psfrags}
    \caption{The situation in the example of Remark~\ref{rmk:f1f2}.}
    \label{fig:remark_f1f2}
  \end{figure}
\end{remark}

%
% two-two node
%
\subsection{Nodes with $n = m = 2$}
\label{sse:n=m=2}

Consider a Riemann solver $\Rr$ for a node $J$ with two incoming and
two outgoing arcs.
In this subsection, we
assume that $(\brho_1,\brho_2,\brho_3,\brho_4)$ denotes an equilibrium for
$\Rr$. Recall that the equilibrium must satisfy $\f 1+\f 2=\f 3+\f 4$.
By symmetry, we may assume also that
\begin{description}
\item[(H1)]  $\brho_1 \le \brho_2$ and $\brho_3 \le \brho_4$.
\end{description}
The results of this subsection are summarized in Table~\ref{tab:equilibrium}.

%
% 4 good
%
\begin{prop}\label{prop:0bad}
  Assume (H1) and that every $\brho_l$ ($l\in\{1,2,3,4\}$) is a good datum.
  \begin{enumerate}
  \item If $\Rr$ satisfies the entropy condition (E1), then
    $\brho_1=\brho_2=\brho_3=\brho_4=\sigma$.

  \item If $\brho_1=\brho_2=\brho_3=\brho_4=\sigma$, then
    $\mathcal F(\bar\rho_1,\brho_2,\brho_3,\brho_4,k)=0$,
    for every $k\in[0,1]$.
  \end{enumerate}
\end{prop}

\begin{proof}
  Since all the data are good, then 
  $\brho_3\le\brho_4\le\sigma\le\brho_1\le\brho_2$.

  If $k\in[\brho_3,\brho_4]$, then \entropia becomes
  \begin{displaymath}
    \f1+\f2-2\ff k\ge \f4-\f3, 
  \end{displaymath}
  which is equivalent to $\ff k\le\f3$. This implies that $\f4=\f3$ and so
  $\brho_3=\brho_4$.

  If $k\in[\brho_1,\brho_2]$, then in the same way we deduce that
  $\brho_1=\brho_2$.

  Finally, if $k\in[\brho_4,\brho_1]$, then (\ref{eq:entropy_RS_E1}),
  coupled with the previous results, becomes
  \begin{displaymath}
    2\f1-2\ff k\ge2\ff k-2\f4,
  \end{displaymath}
  which is equivalent to $\ff k\le\f1$. Therefore $\brho_1=\sigma$ and the
  conclusion follows.
\end{proof}

%
% 3 good
%
\begin{prop}\label{prop:1bad}
  Assume (H1) and that the equilibrium $(\brho_1,\brho_2,\brho_3,\brho_4)$ for
  $\Rr$ is composed by three good data and one bad datum.
  \begin{enumerate}
  \item Assume that the bad datum is in an incoming arc,
    say $\bar\rho_1<\sigma$.\\
    If $\Rr$ satisfies (E1), then $\brho_2=\sigma$ and both 
    $\brho_3$ and $\brho_4$ belong to $[\brho_1,\sigma]$.\\
    If $\brho_2=\sigma$ and both 
    $\brho_3$ and $\brho_4$ belong to $[\brho_1,\sigma]$, then
    $\entropy{\srho,k}$ for every $k\in[0,1]$.

  \item Assume that the bad datum is in an outgoing arc,
    say $\bar \rho_4 > \sigma$.\\
    If $\Rr$ satisfies (E1), then $\brho_3 = \sigma$ and both
    $\brho_1$ and $\brho_2$ belong to $[\sigma,\brho_4]$.\\
    If $\brho_3 = \sigma$ and both $\brho_1$ and $\brho_2$ belong to
    $[\sigma,\brho_4]$, then $\entropy{\srho,k}$ for every $k \in [0,1]$.
  \end{enumerate}
\end{prop}

\begin{proof}
  First assume that the bad datum is in an incoming arc
  and the Riemann solver satisfies the entropy condition (E1).
  Without loss of
  generality, suppose that $\brho_1<\sigma$, $\brho_2\ge\sigma$ and
  $\brho_3\le\brho_4\le\sigma$. We have three possibilities.
  \begin{description}
  \item[(a)] $\brho_1\le\brho_3\le\brho_4$.

  \item[(b)] $\brho_3\le\brho_1\le\brho_4$.

  \item[(c)] $\brho_3\le\brho_4\le\brho_1$.
  \end{description}

  Consider the case \textbf{(a)}. If $k\in[\brho_1,\brho_3]$, then \entropia
  becomes
  \begin{displaymath}
    \f2-\f1\ge\f3+\f4-2\ff k,
  \end{displaymath}
  equivalent to $\ff k\ge\f1$, which is true.\\
  If $k\in[\brho_4,\brho_2]$, then \entropia becomes
  \begin{displaymath}
    \f2-\f1\ge2\ff k-\f3-\f4,
  \end{displaymath}
  equivalent to $\f2\ge\ff k$, which implies that $\brho_2=\sigma$.\\
  If $k\in[\brho_3,\brho_4]$, then \entropia reads
  \begin{displaymath}
    \f2-\f1\ge\f4-\f3,
  \end{displaymath}
  equivalent to $\ff\sigma\ge\f4$, which is true.

  Consider the case \textbf{(b)}. If $k \in [\brho_3,\brho_1]$, then \entropia
  reads
  \begin{displaymath}
    \f1+\f2-2\ff k\ge\f4-\f3,
  \end{displaymath}
  equivalent to $\f3\ge\ff k$. This implies that $\brho_3=\brho_1$ and so
  we are in the case \textbf{(a)}.

  Consider the case \textbf{(c)}. If $k\in[\brho_3,\brho_4]$, then \entropia
  becomes
  \begin{displaymath}
    \f1+\f2-2\ff k\ge\f4-\f3,
  \end{displaymath}
  equivalent to $\f3\ge\ff k$. This implies that $\brho_3=\brho_4$.\\
  If $k\in[\brho_4,\brho_1]$, then \entropia reads
  \begin{displaymath}
    \f1+\f2-2\ff k\ge2\ff k-2\f4,
  \end{displaymath}
  i.e. $\f4\ge\ff k$. This implies that $\brho_4=\brho_1$ and so we have
  a contradiction since, by case \textbf{(a)},
  $\brho_1=\brho_3=\brho_4<\sigma=\brho_2$ and so $\f1+\f2\ne\f3+\f4$.

  The second statement in the case the bad datum is in an incoming arc
  easily follows.

  Assume now that the bad datum is in an outgoing arc
  and that the Riemann solver satisfies the entropy condition (E1).
  Without loss of
  generality, suppose that $\brho_3 \le \sigma$, $\brho_4 > \sigma$ and
  $\sigma \le \brho_1 \le \brho_2$. We have three possibilities.
  \begin{description}
  \item[(a)] $\brho_1\le\brho_2\le\brho_4$.

  \item[(b)] $\brho_1\le\brho_4\le\brho_2$.

  \item[(c)] $\brho_4\le\brho_1\le\brho_2$.
  \end{description}

  Consider the case \textbf{(a)}. If $k \in [\brho_3,\brho_1]$, then \entropia
  becomes
  \begin{displaymath}
    \f1+\f2-2\ff k\ge\f4-\f3,
  \end{displaymath}
  i.e. $\f3 \ge \ff k$. This implies that $\brho_3 = \sigma$.\\
  If  $k \in [\brho_2,\brho_4]$, then (\ref{eq:entropy_RS_E1}) becomes
  \begin{displaymath}
    2\ff k-\f1-\f2\ge\f4-\f3,
  \end{displaymath}
  equivalent to $\ff k \ge \f4$, which is true.\\
  If  $k\in[\brho_1,\brho_2]$, then (\ref{eq:entropy_RS_E1}) becomes
  \begin{displaymath}
    \f2-\f1\ge\f4-\f3,
  \end{displaymath}
  equivalent to $f(\sigma) = \f3 \ge \f1$, which is true.

  Consider the case \textbf{(b)}. If $k \in [\brho_4,\brho_2]$, then \entropia
  reads
  \begin{displaymath}
    \f2 - \f1 \ge 2 \ff k - \f4 - \f3,
  \end{displaymath}
  which is equivalent to $\f2 \ge \ff k$. Thus we deduce that $\brho_2=\brho_4$
  and so we are in the case \textbf{(a)}.

  Consider the case \textbf{(c)}. If $k \in [\brho_1,\brho_2]$, then \entropia
  reads
  \begin{displaymath}
    \f2-\f1\ge2\ff k-\f4-\f3,
  \end{displaymath}
  equivalent to $\f2\ge\ff k$. This implies that $\brho_1=\brho_2$.\\
  If $k \in [\brho_3,\brho_4]$, then (\ref{eq:entropy_RS_E1}) reads
  \begin{displaymath}
    \f1+\f2-2\ff k\ge\f4-\f3,
  \end{displaymath}
  i.e. $\f3\ge\ff k$ and so $\brho_3=\sigma$.
  Therefore $\brho_1=\brho_2=\brho_3=\brho_4=\sigma$, which is a contradiction.

  The second statement of the item~2 of the Proposition easily follows.
  The proof is finished.
\end{proof}

%
% 2 good
%
\begin{prop}\label{prop:2bad}
  Assume (H1) and that the equilibrium $(\brho_1,\brho_2,\brho_3,\brho_4)$ for
  $\Rr$ is composed by two good and two bad data.
  \begin{enumerate}
  \item Assume that $\brho_2<\sigma$, i.e. the bad data
    are both in the incoming arcs.\\
    If the Riemann solver $\Rr$ satisfies the entropy condition (E1),
    then $\brho_1\le\brho_3\le\brho_4\le\brho_2$.\\
    If $\brho_1\le\brho_3\le\brho_4\le\brho_2 < \sigma$, then
    $\entropy{\srho,k}$ for every $k\in[0,1]$.

  \item Assume that $\brho_3>\sigma$, i.e. the bad data are
    in the outgoing arcs.\\
    If the Riemann solver $\Rr$ satisfies the entropy condition (E1),
    then $\brho_3\le\brho_1\le\brho_2\le\brho_4$.\\
    If $\sigma < \brho_3 \le \brho_1 \le \brho_2 \le \brho_4$, then
    $\entropy{\srho,k}$ for every $k\in[0,1]$.

  \item Assume that $\brho_1<\sigma<\brho_4$, i.e. the bad data
    are in the arcs $I_1$ and $I_4$.\\
    If the Riemann solver $\Rr$ satisfies the entropy condition (E1),
    then $\brho_1\le\brho_3\le\sigma\le\brho_2\le\brho_4$.\\
    If $\brho_1\le\brho_3\le\sigma\le\brho_2\le\brho_4$, then
    $\entropy{\srho,k}$ for every $k\in[0,1]$.
  \end{enumerate}
\end{prop}

\begin{proof}
  Assume that $\brho_2<\sigma$ and that the Riemann solver satisfies the
  entropy condition (E1). Since there are exactly two bad data, then
  $\brho_4\le\sigma$. The conservation of mass at $J$ implies that we have
  the following possibilities.
  \begin{description}
  \item[(a)] $\brho_1\le\brho_3\le\brho_4\le\brho_2$.

  \item[(b)] $\brho_3\le\brho_1\le\brho_2\le\brho_4$.
  \end{description}

  Consider the case \textbf{(a)}. If $k\in[\brho_1,\brho_3]$, then \entropia
  reads
  \begin{displaymath}
    \f2-\f1\ge\f3+\f4-2\ff k,
  \end{displaymath}
  equivalent to $\ff k\ge\f1$, which is true.\\
  If $k\in[\brho_3,\brho_4]$, then~(\ref{eq:entropy_RS_E1}) becomes
  \begin{displaymath}
    \f2-\f1\ge\f4-\f3, 
  \end{displaymath}
  which clearly holds.\\
  If $k\in[\brho_4,\brho_2]$, then \entropia reads
  \begin{displaymath}
    \f2-\f1\ge2\ff k-\f3-\f4,
  \end{displaymath}
  equivalent to $\f2\ge\ff k$, which is true.

  Consider the case \textbf{(b)}. If $k\in[\brho_3,\brho_1]$, then \entropia
  reads
  \begin{displaymath}
    \f1+\f2-2\ff k\ge\f4-\f3,
  \end{displaymath}
  equivalent to $\f3\ge\ff k$. This implies that $\brho_1=\brho_3$ and
  consequently $\brho_2=\brho_4$.
  The second statement of the item~1 of the Proposition easily follows.

  Assume now that $\brho_3>\sigma$ and the Riemann solver satisfies the
  entropy condition (E1).
  Consequently $\brho_1\ge\sigma$. Since $\f1+\f2=\f3+\f4$,
  we have the following possibilities.
  \begin{description}
  \item[(a)] $\brho_3\le\brho_1\le\brho_2\le\brho_4$.

  \item[(b)] $\brho_1\le\brho_3\le\brho_4\le\brho_2$.
  \end{description}

  Consider the case \textbf{(a)}. If $k\in[\brho_3,\brho_1]$, then \entropia
  reads
  \begin{displaymath}
    \f1+\f2-2\ff k\ge\f4-\f3,
  \end{displaymath}
  equivalent to $\f3\ge\ff k$, which is true.\\
  If $k\in[\brho_1,\brho_2]$, then \entropia
  reads
  \begin{displaymath}
    \f2-\f1\ge\f4-\f3,
  \end{displaymath}
  which clearly holds.\\
  If $k\in[\brho_2,\brho_4]$, then~(\ref{eq:entropy_RS_E1}) reads
  \begin{displaymath}
    2\ff k-\f1-\f2\ge\f4-\f3,
  \end{displaymath}
  equivalent to $\ff k\ge\f4$, which is true.

  Consider the case \textbf{(b)}. If $k\in[\brho_1,\brho_3]$, then \entropia
  reads
  \begin{displaymath}
    \f2-\f1\ge\f3+\f4-2\ff k,
  \end{displaymath}
  equivalent to $\ff k\ge\f1$. This implies $\brho_1=\brho_3$ and so
  $\brho_2=\brho_4$.
  The second statement of the item~2 of the Proposition easily follows.

  Assume now $\brho_1<\sigma<\brho_4$, i.e. the bad data are in the arcs
  $I_1$ and $I_4$, and that the Riemann solver satisfies the entropy 
  condition (E1). We have the following possibilities.
  \begin{description}
  \item[(a)] $\brho_1\le\brho_3\le\sigma\le\brho_2\le\brho_4$.

  \item[(b)] $\brho_3\le\brho_1<\sigma<\brho_4\le\brho_2$.
  \end{description}

  Consider the case \textbf{(a)}. If $k\in[\brho_1,\brho_3]$, then \entropia
  reads
  \begin{displaymath}
    \f2-\f1\ge\f3+\f4-2\ff k,
  \end{displaymath}
  equivalent to $\ff k\ge\f1$, which is true.\\
  If $k\in[\brho_3,\brho_2]$, then (\ref{eq:entropy_RS_E1}) becomes
  \begin{displaymath}
    \f2-\f1\ge\f4-\f3,
  \end{displaymath}
  equivalent to $\f3\ge\f1$, which is true.\\
  If $k\in[\brho_2,\brho_4]$, then \entropia becomes
  \begin{displaymath}
    2\ff k-\f1-\f2\ge\f4-\f3,
  \end{displaymath}
  equivalent to $\ff k\ge\f4$, which is true.

  Consider the case \textbf{(b)}. If $k\in[\brho_3,\brho_1]$, then \entropia
  reads
  \begin{displaymath}
    \f1+\f2-2\ff k\ge\f4-\f3,
  \end{displaymath}
  equivalent to $\f3\ge\ff k$. This implies $\brho_1=\brho_3$ and so
  $\brho_2=\brho_4$.
  The second statement of the item~3 of the Proposition easily follows.

  The proof is finished.
\end{proof}

%
% 1 good
%
\begin{prop}\label{prop:3bad}
  Assume (H1) and that the equilibrium $(\brho_1,\brho_2,\brho_3,\brho_4)$ for
  $\Rr$ is composed by three bad data and one good datum.
  \begin{enumerate}
  \item Assume that $\brho_2\ge\sigma$, i.e. the good datum is
    in an incoming arc.\\
    If the Riemann solver satisfies the entropy condition (E1),
    then $\brho_1<\sigma$, $\brho_3>\sigma$, $\brho_2\le\brho_4$ and
    $\f1\le\max\left\{\f2,\f3\right\}$.\\
    If $\brho_1<\sigma$, $\brho_3>\sigma$, $\brho_2\le\brho_4$ and
    $\f1\le\max\left\{\f2,\f3\right\}$, then
    $\entropy{\srho,k}$ for every $k\in[0,1]$.

  \item Assume that $\brho_3\le\sigma$, i.e. the good datum
    is in an outgoing arc.\\
    If the Riemann solver satisfies the entropy condition (E1),
    then $\brho_2<\sigma$, $\brho_4>\sigma$, $\brho_3\ge\brho_1$ and
    $\f4\le\max\left\{\f2,\f3\right\}$.
    If $\brho_2<\sigma$, $\brho_4>\sigma$, $\brho_3\ge\brho_1$ and
    $\f4\le\max\left\{\f2,\f3\right\}$, then
    $\entropy{\srho,k}$ for every $k\in[0,1]$.
  \end{enumerate}
\end{prop}

\begin{proof}
  Assume first that $\brho_2\ge\sigma$ and that the Riemann solver satisfies
  the entropy condition (E1). We easily deduce that
  $\brho_1<\sigma<\brho_3\le\brho_4$. We have the following possibilities.
  \begin{description}
  \item[(a)] $\brho_1<\sigma\le\brho_2\le\brho_3\le\brho_4$.

  \item[(b)] $\brho_1<\sigma<\brho_3\le\brho_2\le\brho_4$.

  \item[(c)] $\brho_1<\sigma<\brho_3\le\brho_4\le\brho_2$.
  \end{description}

  Consider the case \textbf{(a)}. If $k\in[\brho_1,\brho_2]$, then \entropia
  reads
  \begin{displaymath}
    \f2-\f1\ge\f3+\f4-2\ff k,
  \end{displaymath}
  equivalent to $\ff k\ge\f1$. This implies that $\f2\ge\f1$.\\
  If $k\in[\brho_2,\brho_3]$, then~(\ref{eq:entropy_RS_E1}) becomes
  \begin{displaymath}
    2\ff k-\f1-\f2\ge\f3+\f4-2\ff k,
  \end{displaymath}
  equivalent to $2\ff k\ge\f3+\f4$, which is true.\\
  If $k\in[\brho_3,\brho_4]$, then \entropia reads
  \begin{displaymath}
    2\ff k-\f1-\f2\ge\f4-\f3,
  \end{displaymath}
  equivalent to $\ff k\ge\f4$, which is true.

  Consider the case \textbf{(b)}. If $k\in[\brho_1,\brho_3]$, then \entropia
  reads
  \begin{displaymath}
    \f2-\f1\ge\f3+\f4-2\ff k,
  \end{displaymath}
  equivalent to $\ff k\ge\f1$. This implies that $\f3\ge\f1$.\\
  If $k\in[\brho_3,\brho_2]$, then~(\ref{eq:entropy_RS_E1}) becomes
  \begin{displaymath}
    \f2-\f1\ge\f4-\f3,
  \end{displaymath}
  equivalent to $\f3\ge\f1$.\\
  If $k\in[\brho_2,\brho_4]$, then \entropia reads
  \begin{displaymath}
    2\ff k-\f1-\f2\ge\f4-\f3,
  \end{displaymath}
  equivalent to $\ff k\ge\f4$, which is true.

  Consider the case \textbf{(c)}. If $k\in[\brho_4,\brho_2]$, then \entropia
  reads
  \begin{displaymath}
    \f2-\f1\ge2\ff k-\f3-\f4,
  \end{displaymath}
  equivalent to $\f2\ge\ff k$. This implies that $\brho_2=\brho_4$ and so we
  are in the case \textbf{(b)}.
  The second statement in the item~1 of the Proposition easily follows.

  %%%%%%%%%%%%%% 
  Assume now that $\brho_3\le\sigma$ and that the Riemann solver
  satisfies the entropy condition (E1). We easily deduce that
  $\brho_1\le\brho_2<\sigma<\brho_4$. We have the following possibilities.
  \begin{description}
  \item[(a)] $\brho_1\le\brho_2\le\brho_3\le\sigma<\brho_4$.

  \item[(b)] $\brho_1\le\brho_3\le\brho_2<\sigma<\brho_4$.

  \item[(c)] $\brho_3\le\brho_1\le\brho_2<\sigma<\brho_4$.
  \end{description}

  Consider the case \textbf{(a)}. If $k\in[\brho_1,\brho_2]$, then \entropia
  reads
  \begin{displaymath}
    \f2-\f1\ge\f3+\f4-2\ff k,
  \end{displaymath}
  equivalent to $\ff k\ge\f1$, which is true.\\
  If $k\in[\brho_2,\brho_3]$, then~(\ref{eq:entropy_RS_E1}) becomes
  \begin{displaymath}
    2\ff k-\f1-\f2\ge\f3+\f4-2\ff k,
  \end{displaymath}
  equivalent to $2\ff k\ge\f1+\f2$, which is true.\\
  If $k\in[\brho_3,\brho_4]$, then \entropia reads
  \begin{displaymath}
    2\ff k-\f1-\f2\ge\f4-\f3,
  \end{displaymath}
  equivalent to $\ff k\ge\f4$. This implies that $\f3\ge\f4$.

  Consider the case \textbf{(b)}. If $k\in[\brho_1,\brho_3]$, then \entropia
  reads
  \begin{displaymath}
    \f2-\f1\ge\f3+\f4-2\ff k,
  \end{displaymath}
  equivalent to $\ff k\ge\f1$, which is true.\\
  If $k\in[\brho_3,\brho_2]$, then~(\ref{eq:entropy_RS_E1}) becomes
  \begin{displaymath}
    \f2-\f1\ge\f4-\f3,
  \end{displaymath}
  equivalent to $\f3\ge\f1$, which is true.\\
  If $k\in[\brho_2,\brho_4]$, then \entropia reads
  \begin{displaymath}
    2\ff k-\f1-\f2\ge\f4-\f3,
  \end{displaymath}
  equivalent to $\ff k\ge\f4$. This implies that $\f2\ge\f4$.

  Consider the case \textbf{(c)}. If $k\in[\brho_3,\brho_1]$, then \entropia
  reads
  \begin{displaymath}
    \f1+\f2-2\ff k\ge\f4-\f3,
  \end{displaymath}
  equivalent to $\f3\ge\ff k$. This implies that $\brho_1=\brho_3$ and so we
  are in the case \textbf{(b)}.
  The second statement in the item~2 of the Proposition easily follows.

  The proof is finished.
\end{proof}

%
% 0 good
%
\begin{prop}\label{prop:4bad}
  Assume (H1) and that the equilibrium $(\brho_1,\brho_2,\brho_3,\brho_4)$ for
  $\Rr$ is composed by four bad data.
  If the Riemann solver satisfies the entropy condition (E1), then
  $\brho_1 \le \brho_2 < \sigma < \brho_3 \le \brho_4$.
  Moreover, if $\brho_1 \le \brho_2 < \sigma < \brho_3 \le \brho_4$,
  then $\entropy{\srho,k}$ for every $k\in[0,1]$.
\end{prop}

\begin{proof}
  It is sufficient to check the entropy condition (E1).
  If $k\in[\brho_1,\brho_2]$, then \entropia reads
  \begin{displaymath}
    \f2-\f1\ge\f3+\f4-2\ff k,
  \end{displaymath}
  equivalent to $\ff k\ge\f1$, which is true.\\
  If $k\in[\brho_2,\brho_3]$, then~(\ref{eq:entropy_RS_E1}) becomes
  \begin{displaymath}
    2\ff k-\f1-\f2\ge\f3+\f4-2\ff k,
  \end{displaymath}
  equivalent to $2\ff k\ge\f3+\f4$, which is true.
  If $k\in[\brho_3,\brho_4]$, then \entropia reads
  \begin{displaymath}
    2\ff k-\f1-\f2\ge\f4-\f3,
  \end{displaymath}
  equivalent to $\ff k\ge\f4$, which is true. This concludes the proof.
\end{proof}

\begin{table}[t]
  \centering
  \begin{tabular}{|c|l|}
    \hline
    \small{Bad data} & \hspace{2.5cm}admissible configurations\\
    \hline \hline
    $0$ & $\brho_1 = \brho_2 = \brho_3 = \brho_4 = \sigma$\\
    \hline
    $1$ & $\brho_1 \le \brho_3 \le \brho_4 \le \sigma = \brho_2$,
    \hspace{.5cm} $\brho_1 < \sigma$\\
    \cline{2-2}
    & $\brho_3 = \sigma \le \brho_1 \le \brho_2 \le \brho_4$,
    \hspace{.5cm} $\brho_4 > \sigma$\\
    \hline \hline
    $2$ & $\brho_1 \le \brho_3 \le \brho_4 \le \brho_2 < \sigma$\\
    \cline{2-2}
    & $\sigma < \brho_3 \le \brho_1 \le \brho_2 \le \brho_4$\\
    \cline{2-2}
    & $\brho_1 \le \brho_3 \le \sigma \le \brho_2 \le \brho_4$,
    \hspace{.5cm} $\brho_1 < \sigma < \brho_4$\\
    \hline \hline
    $3$ & $\brho_1 < \sigma < \brho_3 \le \brho_4$,\,\,
    $\sigma \le \brho_2 \le \brho_4$,\,\,
    $\f1 \le \max\left\{ \f2, \f3 \right\}$\\
    \cline{2-2}
    & $\brho_1 \le \brho_2 < \sigma < \brho_4$,\,\,
    $\brho_1 \le \brho_3 \le \sigma$,\,\,
    $\f4 \le \max\left\{ \f2, \f3 \right\}$\\
    \hline \hline
    $4$ & $\brho_1 \le \brho_2 < \sigma < \brho_3 \le \brho_4$\\
    \hline
  \end{tabular}
  \caption{All the possible configurations for an equilibrium
  $(\srho)$ of a $\Rr$ satisfying the entropy condition (E1).
  By symmetry, we assume that $\brho_1 \le \brho_2$ and $\brho_3 \le \brho_4$,
  i.e. (H1) holds.}
  \label{tab:equilibrium}
\end{table}

\begin{remark}
  \label{rmk:RS_E1}
  Note that there exist Riemann solvers satisfying the consistency condition
  and the entropy condition (E1). Here we construct a Riemann
  solver $\Rr$ with such properties.\\
  Consider an initial condition
  $(\rho_{1,0}, \rho_{2,0}, \rho_{3,0}, \rho_{4,0})$.
  Denote with $(\hat \rho_1, \hat \rho_2, \hat \rho_3, \hat \rho_4)$
  the image of the initial condition through $\Rr$, i.e.
  \begin{displaymath}
    (\hat \rho_1, \hat \rho_2, \hat \rho_3, \hat \rho_4) = 
    \Rr (\rho_{1,0}, \rho_{2,0}, \rho_{3,0}, \rho_{4,0})
  \end{displaymath}
  If $h$ is the number of bad initial data, then we define $\Rr$
  according to the following possibilities.
  \begin{description}
  \item[$h=0$.] We put
    $\hat \rho_1 = \hat \rho_2 = \hat \rho_3 = \hat \rho_4 = \sigma$.
    By Proposition~\ref{prop:0bad}, this provides an entropy admissible
    equilibrium. Moreover
    \begin{displaymath}
      \Rr \left( \Rr (\rho_{1,0}, \rho_{2,0}, \rho_{3,0}, \rho_{4,0}) \right)
      = \Rr (\rho_{1,0}, \rho_{2,0}, \rho_{3,0}, \rho_{4,0}).
    \end{displaymath}

  \item[$h=1$.] Let $\bar l\in\{1,2,3,4\}$ be such that $\rho_{\bar l,0}$
    is a bad datum.
    We have two possibilities: $\bar l \le 2$ or $\bar l \ge 3$.\\
    Assume first $\bar l\le 2$. We put $\hat \rho_{\bar l} = \rho_{\bar l,0}$
    and $\hat \rho_l = \sigma$ for $l \in\{1,2\}$, $l \ne \bar l$.
    Moreover we define $\hat \rho_3 = \hat \rho_1$ and
    $\hat \rho_4 = \hat \rho_2$.\\
    Assume now $\bar l\ge 3$. We put $\hat \rho_{\bar l} = \rho_{\bar l,0}$
    and $\hat \rho_l = \sigma$ for $l \in\{3,4\}$, $l \ne \bar l$.
    Moreover we define $\hat \rho_1 = \hat \rho_3$ and
    $\hat \rho_2 = \hat \rho_4$.\\
    By Proposition~\ref{prop:1bad}, these solutions provide entropy admissible
    equilibria. Moreover
    \begin{displaymath}
      \Rr \left( \Rr (\rho_{1,0}, \rho_{2,0}, \rho_{3,0}, \rho_{4,0}) \right)
      = \Rr (\rho_{1,0}, \rho_{2,0}, \rho_{3,0}, \rho_{4,0}).
    \end{displaymath}

  \item[$h=2$.] Let $l_1, l_2 \in \{1,2,3,4\}$, $\l_1\ne l_2$, be such that
    $\rho_{l_1,0}$ and $\rho_{l_2,0}$ are bad data.
    We have three different possibilities.\\
    Assume first that $l_1, l_2\in\{1,2\}$. In this case we put
    $\hat \rho_{l_1} = \rho_{l_1,0}$, $\hat \rho_{l_2} = \rho_{l_2,0}$,
    $\hat \rho_3 = \hat \rho_1$ and $\hat \rho_4 = \hat \rho_2$.\\
    Assume now that $l_1, l_2\in\{3,4\}$. In this case we put
    $\hat \rho_{l_1} = \rho_{l_1,0}$, $\hat \rho_{l_2} = \rho_{l_2,0}$,
    $\hat \rho_1 = \hat \rho_3$ and $\hat \rho_2 = \hat \rho_4$.\\
    Consider finally the last case. For simplicity suppose that
    $l_1 = 1$ and $l_2 = 4$.
    We define $\hat \rho_{l_1} = \rho_{l_1,0}$,
    $\hat \rho_{l_2} = \rho_{l_2,0}$, $\hat \rho_2 = \hat \rho_{l_2}$
    and $\hat \rho_3 = \hat \rho_{l_1}$.
    By Proposition~\ref{prop:2bad}, these solutions provide entropy admissible
    equilibria. Moreover
    \begin{displaymath}
      \Rr \left( \Rr (\rho_{1,0}, \rho_{2,0}, \rho_{3,0}, \rho_{4,0}) \right)
      = \Rr (\rho_{1,0}, \rho_{2,0}, \rho_{3,0}, \rho_{4,0}).
    \end{displaymath}

  \item[$h=3$.] Let $\bar l\in\{1,2,3,4\}$ be such that $\rho_{\bar l,0}$
    is a good datum.
    We have two possibilities: $\bar l \le 2$ or $\bar l \ge 3$.\\
    Assume first $\bar l \le 2$; say $\bar l = 2$ for simplicity.\\
    If $f(\rho_{3,0}) + f(\rho_{4,0}) - f(\rho_{1,0}) \in
    \left[\min\left\{  f(\rho_{3,0}), f(\rho_{4,0})\right\},
      f(\sigma)\right]$, then
    we put $\hat \rho_l = \rho_{l,0}$ for every $l \in \{1,2,3,4\}$,
    $l \ne \bar l$ and $\hat \rho_{\bar l}\in [\sigma, 1]$ such that
    $f(\hat \rho_{2}) = f(\rho_{3,0}) + f(\rho_{4,0}) - f(\rho_{1,0})$.\\
    If $f(\rho_{3,0}) + f(\rho_{4,0}) - f(\rho_{1,0}) > f(\sigma)$
    and $f(\rho_{3,0}) \ge f(\rho_{4,0})$, then
    $\hat \rho_1 = \hat \rho_3 =\rho_{1,0}$ and
    $\hat \rho_2 = \hat \rho_4 =\rho_{4,0}$.\\
    If $f(\rho_{3,0}) + f(\rho_{4,0}) - f(\rho_{1,0}) > f(\sigma)$
    and $f(\rho_{3,0}) < f(\rho_{4,0})$, then
    $\hat \rho_2 = \hat \rho_3 =\rho_{3,0}$ and
    $\hat \rho_1 = \hat \rho_4 =\rho_{1,0}$.\\
    If $f(\rho_{3,0}) + f(\rho_{4,0}) - f(\rho_{1,0}) < 
    \min\left\{ f(\rho_{3,0}), f(\rho_{4,0})\right\}$, then
    $\hat \rho_2 = \hat \rho_4 =\rho_{3,0}$ and
    $\hat \rho_1 = \hat \rho_3 =\rho_{4,0}$.

    Assume now $\bar l \ge 3$; say $\bar l = 3$ for simplicity.\\
    If $f(\rho_{1,0}) + f(\rho_{2,0}) - f(\rho_{4,0}) \in
    \left[\min\left\{  f(\rho_{1,0}), f(\rho_{2,0})\right\},
      f(\sigma)\right]$, then
    we put $\hat \rho_l = \rho_{l,0}$ for every $l \in \{1,2,3,4\}$,
    $l \ne \bar l$ and $\hat \rho_{\bar l}\in [0,\sigma]$ such that
    $f(\hat \rho_{3}) = f(\rho_{1,0}) + f(\rho_{2,0}) - f(\rho_{4,0})$.\\
    If $f(\rho_{1,0}) + f(\rho_{2,0}) - f(\rho_{4,0}) > f(\sigma)$
    and $f(\rho_{1,0}) \ge f(\rho_{2,0})$, then
    $\hat \rho_1 = \hat \rho_4 =\rho_{4,0}$ and
    $\hat \rho_2 = \hat \rho_3 =\rho_{2,0}$.\\
    If $f(\rho_{1,0}) + f(\rho_{2,0}) - f(\rho_{4,0}) > f(\sigma)$
    and $f(\rho_{2,0}) > f(\rho_{1,0})$, then
    $\hat \rho_2 = \hat \rho_4 =\rho_{4,0}$ and
    $\hat \rho_1 = \hat \rho_3 =\rho_{1,0}$.\\
    If $f(\rho_{1,0}) + f(\rho_{2,0}) - f(\rho_{4,0}) < 
    \min\left\{ f(\rho_{1,0}), f(\rho_{2,0})\right\}$, then
    $\hat \rho_2 = \hat \rho_4 =\rho_{2,0}$ and
    $\hat \rho_1 = \hat \rho_3 =\rho_{1,0}$.

    By Propositions~\ref{prop:2bad} and~\ref{prop:3bad},
    these solutions provide entropy admissible
    equilibria. Moreover
    \begin{displaymath}
      \Rr \left( \Rr (\rho_{1,0}, \rho_{2,0}, \rho_{3,0}, \rho_{4,0}) \right)
      = \Rr (\rho_{1,0}, \rho_{2,0}, \rho_{3,0}, \rho_{4,0}).
    \end{displaymath}

  \item[$h=4$.] We have some different cases.
    Assume first that $f(\rho_{1,0}) + f(\rho_{2,0}) 
    = f(\rho_{3,0}) + f(\rho_{4,0})$. We put
    $\hat \rho_1 = \rho_{1,0}$, $\hat \rho_2 = \rho_{2,0}$,
    $\hat \rho_3 = \rho_{3,0}$ and $\hat \rho_4 = \rho_{4,0}$.\\
    Assume now that $f(\rho_{1,0}) + f(\rho_{2,0}) < 
    f(\rho_{3,0}) + f(\rho_{4,0})$. For simplicity suppose that
    $f(\rho_{1,0}) \le f(\rho_{2,0})$ and $f(\rho_{3,0}) \ge f(\rho_{4,0})$.\\
    If $f(\rho_{4,0}) > f(\rho_{2,0})$, then we put
    $\hat \rho_1 = \hat \rho_3 = \rho_{1,0}$ and
    $\hat \rho_2 = \hat \rho_4 = \rho_{2,0}$.\\
    If $f(\rho_{4,0}) \le f(\rho_{2,0})$, then we put
    $\hat \rho_1 = \rho_{1,0}$, $\hat \rho_2 = \rho_{2,0}$,
    $\hat \rho_4 = \rho_{4,0}$ and $\hat \rho_3 \in [0,\sigma]$
    such that $f(\hat \rho_3) = f(\hat \rho_1) + f(\hat \rho_2) 
    - f(\hat \rho_4)$. 

    Assume finally that $f(\rho_{1,0}) + f(\rho_{2,0}) > 
    f(\rho_{3,0}) + f(\rho_{4,0})$.
    For simplicity suppose that
    $f(\rho_{1,0}) \le f(\rho_{2,0})$ and $f(\rho_{3,0}) \ge f(\rho_{4,0})$.\\
    If $f(\rho_{1,0}) > f(\rho_{3,0})$, then we put
    $\hat \rho_1 = \hat \rho_3 = \rho_{3,0}$ and
    $\hat \rho_2 = \hat \rho_4 = \rho_{4,0}$.\\
    If $f(\rho_{1,0}) \le f(\rho_{3,0})$, then we put
    $\hat \rho_1 = \rho_{1,0}$, $\hat \rho_3 = \rho_{3,0}$,
    $\hat \rho_4 = \rho_{4,0}$ and $\hat \rho_2 \in [\sigma, 1]$
    such that $f(\hat \rho_2) = f(\hat \rho_3) + f(\hat \rho_4) 
    - f(\hat \rho_1)$. 

    By Propositions~\ref{prop:2bad}, \ref{prop:3bad} and~\ref{prop:4bad},
    these solutions provide entropy admissible
    equilibria. Moreover
    \begin{displaymath}
      \Rr \left( \Rr (\rho_{1,0}, \rho_{2,0}, \rho_{3,0}, \rho_{4,0}) \right)
      = \Rr (\rho_{1,0}, \rho_{2,0}, \rho_{3,0}, \rho_{4,0}).
    \end{displaymath}
  \end{description}
\end{remark}

\begin{remark}
  Another example of Riemann solver satisfying the entropy
  condition (E1) for a node with two incoming and two outgoing arcs 
  is a particular case of the Riemann solver $\Rr_2$,
  defined in Section~\ref{ssec:rs2};
  see Proposition~\ref{prop:rs2_E1}.

  The Riemann solver $\Rr$, constructed in Remark~\ref{rmk:RS_E1},
  differs from the Riemann solver $\Rr_2$. The key difference is that
  a permutation of initial data in incoming (resp. outgoing) arcs
  influences the solution in outgoing (resp. incoming) arcs in the case
  of $\Rr$, but not in the case of $\Rr_2$.\\
  Consider the following example. Let $f(\rho) = 4 \rho (1 - \rho)$
  be the flux. Assume that $\left( \frac14, \frac34, \frac14, \frac14 \right)$
  and $\left( \frac34, \frac14, \frac14, \frac14 \right)$ are two initial
  conditions. In both cases, we have only one bad datum and so,
  using the notation of Remark~\ref{rmk:RS_E1}, $h = 1$. Hence we deduce
  \begin{displaymath}
    \Rr \left( \frac14, \frac34, \frac14, \frac14 \right) =
    \left( \frac14, \frac12, \frac14, \frac12 \right)
  \end{displaymath}
  and
  \begin{displaymath}
    \Rr \left( \frac34, \frac14, \frac14, \frac14 \right) =
    \left( \frac12, \frac14, \frac12, \frac14 \right),
  \end{displaymath}
  while
  \begin{displaymath}
    \Rr_2 \left( \frac14, \frac34, \frac14, \frac14 \right) =
    \Rr_2 \left( \frac34, \frac14, \frac14, \frac14 \right);
  \end{displaymath}
  see Section~\ref{ssec:rs2}.
\end{remark}

%
%
% examples
%
%
\section{Examples}
\label{se:examples}

This Section deals with some examples of Riemann solvers, introduced in
literature for describing car and data traffic. For each of them, we
analyze the entropy conditions (E1) and (E2).
First we need some notation.

Consider the set
\begin{equation}
  \label{eq:calA}
  \mathcal A:=\left\{
    \begin{array}{ll}
      A=\{a_{ji}\}_{\substack{i=1,\ldots,n\\ j=n+1,\ldots,n+m}}: &
      \begin{array}{l}
        0 < a_{ji} <   1\,\,      \forall i,j,\\
        \sum\limits_{j=n +1}^{n+m} a_{ji} =1\,\,\forall i
      \end{array}    
    \end{array}
  \right\}.
\end{equation}

Let $\{e_1,\ldots,e_n\}$ be the canonical basis of $\R^n$.
For every $i=1,\ldots,n$, we denote $H_i=\{e_i\}^\bot$.
If $A\in\mathcal A$, then we write,
for every $j=n+1,\ldots,n+m$, $a_j=(a_{j1},\ldots,a_{jn})\in\R^n$
and $H_j=\{a_j\}^\bot$.
Let $\KK$ be the set of indices ${\bf k}=(k_1,...,k_\ell)$,
$1\leq\ell\leq n-1$, such that
$0\leq k_1<k_2<\cdots<k_\ell\leq n+m$ and for every ${\bf k}\in\KK$ define
\begin{equation*}
  H_{\bf k}=\bigcap\limits_{h=1}^\ell H_{k_h}.
\end{equation*}
Writing ${\bf 1}=(1,\ldots,1)\in\R^n$ and
following \cite{C-G-P} we define the set
\begin{equation}
  \label{eq:frakn}
  \mathfrak N:=\left\{A\in\mathcal A:{\bf 1}\notin H_{\bf k}^\bot\,
    \textrm{ for every } {\bf k}\in\KK
\right\}\,.
\end{equation}
Notice that, if $n> m$, then $\mathfrak N=\emptyset$.
The matrices of $\mathfrak N$ will give rise of a unique solution
to the Riemann problem at $J$.

For later use, define the set
\begin{equation}
  \label{eq:theta}
  \Theta=\left\{\boldsymbol{\theta}=\left(\theta_1,\ldots,\theta_{n+m}\right)
    \in\R^{n+m}:\,
    \begin{array}{c}
      \theta_1>0,\cdots,\theta_{n+m}>0,\vspace{.2cm}\\
      \sum_{i=1}^n\theta_i=\sum_{j=n+1}^{n+m}\theta_j=1
    \end{array}
  \right\}\,.
\end{equation}

%
% Riemann solver 1
%
\subsection{Riemann Solver $\mathcal{RS}_1$}
In this subsection, we consider the Riemann solver introduced for car
traffic in~\cite{C-G-P}. The construction can be done in the following way.

\begin{enumerate}
\item Fix a matrix $A \in \mathfrak N$ and consider the closed,
  convex and not empty set
  \begin{equation}
    \label{eq:omega}
    \Omega=\left\{
      (\gamma_1,\cdots,\gamma_n)\in\prod_{i=1}^n\Omega_i:
      A\cdot (\gamma_1,\cdots,\gamma_n)^T\in\prod_{j=n+1}^{n+m}\Omega_j
    \right\}\,.
  \end{equation}
  
\item Find the point $(\bar\gamma_1,\ldots,\bar\gamma_n)\in\Omega$
  which maximizes the function 
  \begin{equation}\label{eq:E}
    E(\gamma_1,\ldots,\gamma_n)=\gamma_1+\cdots+\gamma_n,
  \end{equation}
  and define $(\bar\gamma_{n+1},\ldots,\bar\gamma_{n+m})^T
  :=A\cdot(\bar\gamma_1,\ldots,\bar\gamma_n)^T$.
  Since $A\in\mathfrak N$, then
  $(\bar\gamma_1,\ldots,\bar\gamma_n)$ is unique.

\item For every $i\in\{1,\ldots,n\}$, define $\bar\rho_i$
  either by $\rho_{i,0}$
  if $f(\rho_{i,0})=\bar\gamma_i$, or by the solution to
  $f(\rho)=\bar\gamma_i$
  such that $\bar\rho_i\ge\sigma$.
  For every $j\in\{n+1,\ldots,n+m\}$,
  define $\bar\rho_j$ either by $\rho_{j,0}$
  if $f(\rho_{j,0})=\bar\gamma_j$, or by the solution
  to $f(\rho)=\bar\gamma_j$
  such that $\bar\rho_j\le\sigma$.
  Finally, define $\Rr_1:[0,1]^{n+m}\to[0,1]^{n+m}$ by
  \begin{equation}\label{eq:rs1_rho}
    \Rr_1(\rho_{1,0},\ldots,\rho_{n+m,0})
    =(\bar\rho_1,\ldots,\bar\rho_n,\bar\rho_{n+1},\ldots,\bar\rho_{n+m})\,.
  \end{equation}
\end{enumerate}

The following result holds.

\begin{lemma}
  The function defined in~(\ref{eq:rs1_rho}) satisfies the consistency
  condition, in the sense of Definition~\ref{def:consistency}.
\end{lemma}

For a proof, see~\cite{C-G-P,gp-book}. We show that this Riemann
solver does not satisfy neither the entropy condition (E1) nor (E2).

\begin{prop}\label{prop:no_E2_rs1}
  The Riemann solver $\Rr_1$ does not satisfy the entropy condition (E2)
  in the sense of Definition~\ref{def:entropy_RS_E2} and, consequently,
  does not satisfy the entropy condition (E1)
  in the sense of Definition~\ref{def:entropy_RS_E1}.
\end{prop}

\begin{proof}
  Consider a node with $2$ incoming and $2$ outgoing arcs,
  the flux function $f(\rho)=4\rho(1-\rho)$, a matrix
  \begin{equation*}
    A=\left(
      \begin{array}{cc}
        \frac13 & \frac12\vspace{.2cm}\\
        \frac23 & \frac12
      \end{array}
    \right)
  \end{equation*}
  and the initial conditions $\rho_{1,0}=\frac34$, $\rho_{2,0}=\frac18$,
  $\rho_{3,0}=\frac{8+\sqrt{34}}{16}$ and $\rho_{4,0}=\frac1{10}$.
  In this case the set $\Omega$ in~(\ref{eq:omega}) is
  \begin{displaymath}
    \left\{(\gamma_1,\gamma_2)\in[0,1]\times\left[0,\frac7{16}\right]:
      0\le\frac{\gamma_1}3+\frac{\gamma_2}2\le\frac{15}{32},\,
      0\le\frac{2\gamma_1}3+\frac{\gamma_2}2\le1\right\}\,;
  \end{displaymath}
  see Figure~\ref{fig:omega_rs1_no_E2}.
  \begin{figure}[h]
    \centering
    \begin{psfrags}
      \psfrag{gamma_1}{$\gamma_1$}
      \psfrag{gamma_2}{$\gamma_2$}
      \psfrag{Insieme Omega}{$\Omega$}
      \includegraphics[width=7cm]{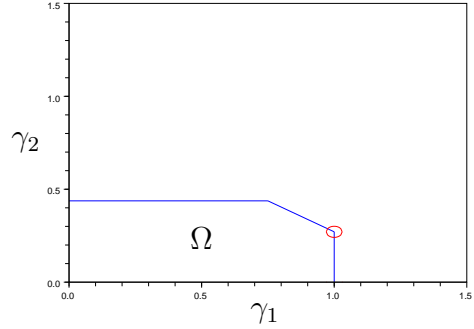}
    \end{psfrags}
    \caption{The set $\Omega$ of Proposition~\ref{prop:no_E2_rs1}.}
    \label{fig:omega_rs1_no_E2}
  \end{figure}
  Therefore we deduce that $\bar\gamma_1=1$, $\bar\gamma_2=\frac{13}{48}$,
  $\bar\gamma_3=\frac{15}{32}$, $\bar\gamma_4=\frac{77}{96}$,
  $\bar\rho_1=\sigma$, $\bar\rho_2>\sigma$, $\bar\rho_3=\rho_{3,0}$ and
  $\bar\rho_4<\sigma$.
  The entropy condition~(\ref{eq:entropy_RS_E2}) in this case becomes
  \begin{displaymath}
    f(\bar\rho_2)-f(\sigma)\ge f(\bar\rho_3)-f(\sigma)+f(\sigma)-f(\bar\rho_4),
  \end{displaymath}
  which is equivalent to
  \begin{displaymath}
    0\le f(\bar\rho_2)-f(\sigma)-f(\bar\rho_3)+f(\bar\rho_4)=
    \frac{13}{48}-1-\frac{15}{32}+\frac{77}{96}=-\frac{19}{48}\,.
  \end{displaymath}
  This concludes the proof.
\end{proof}

The maximization of the function $E$ over $\Omega$, which defines the
Riemann solver $\Rr_1$, is, however, in connection with the maximization
of the entropy $\mathcal F$. In order to explain this fact, let
us introduce some notations.\\
Given $\Omega$ in \eqref{eq:omega}, define
\begin{equation}
  \label{eq:PHI}
  \Phi=\!\left\{
    (\rho_1,\ldots,\rho_{n+m})\in\prod_{l=1}^{n+m}\Phi_l:\!\!
    \begin{array}{l}
      \left(f(\rho_1),\ldots,f(\rho_n)\right)\in\Omega,\\
      \left(
        \begin{array}{c}
          f(\rho_{n+1})\\
          \vdots\\
          f(\rho_{n+m})
        \end{array}
      \right)=A\cdot
      \left(
        \begin{array}{c}
          f(\rho_1)\\
          \vdots\\
          f(\rho_n)
        \end{array}
      \right)\!
    \end{array}
  \right\}
\end{equation}
and the functional
\begin{equation}
  \label{eq:cal_G}
  \begin{array}{ccc}
    \mathcal G: \Phi & \longrightarrow & \R\\
    (\rho_1, \ldots, \rho_{n+m}) & \longmapsto & 
    \mathcal F(\rho_1, \ldots, \rho_{n+m}, \sigma),
  \end{array}
\end{equation}
which is the restriction of $\mathcal F$ on $\Phi \times \{ \sigma \}$.
Note that the set $\Phi$ consists in all the possible solutions at $J$
satisfying Definition~\ref{def:Riemann_solver} and the distribution rule,
determined by the matrix $A \in \mathfrak N$.
It is easy to see that there exists a one to one
correspondence between $\Omega$ and $\Phi$.\\
For every $\mathcal H\s\{1,\ldots,n+m\}$ of cardinality $h$,
with $0\le h\le n-1$, define
\begin{equation}
  \label{eq:Omega_h}
  \Omega_{\mathcal H}=\!\left\{
    (\gamma_1,\ldots,\gamma_n)\in\prod_{i=1}^n\Omega_i:\!
    \begin{array}{l}
      (\gamma_{n+1},\ldots,\gamma_{n+m})^T\!
      =\! A\!\cdot\!(\gamma_1,\ldots,\gamma_n)^T,\\
      (\gamma_{n+1},\ldots,\gamma_{n+m})\in\prod_{j=n+1}^{n+m}\Omega_j,\\
      \gamma_l=\max \Omega_l\quad\textrm{ if }\quad l\in\mathcal H,\\
      \gamma_l<\max \Omega_l\quad\textrm{ if }\quad l\not\in\mathcal H,\\
    \end{array}
  \right\}
\end{equation}
and
\begin{equation}
  \label{eq:PHI_h}
  \Phi_{\mathcal H}=\!\left\{
    (\rho_1,\ldots,\rho_{n+m})\in\prod_{l=1}^{n+m}\Phi_l:\!\!
    \begin{array}{l}
      \left(f(\rho_1),\ldots,f(\rho_n)\right)\in\Omega_{\mathcal H},\\
      \left(
        \begin{array}{c}
          f(\rho_{n+1})\\
          \vdots\\
          f(\rho_{n+m})
        \end{array}
      \right)=A\cdot
      \left(
        \begin{array}{c}
          f(\rho_1)\\
          \vdots\\
          f(\rho_n)
        \end{array}
      \right)\!
    \end{array}
  \right\}.
\end{equation}
Notice that $\OmegaH$ and $\PhiH$ depend on the initial condition
$(\rho_{1,0}, \ldots, \rho_{n+m,0})$ and on the matrix $A \in \mathfrak N$.
There is a one to one correspondence between $\OmegaH$
and $\PhiH$, given by the one-to-one function
\begin{displaymath}
  \begin{array}{ccc}
    \PhiH & \longrightarrow & \OmegaH\\
    (\rho_1, \ldots, \rho_{n+m}) & \longmapsto &
    (f(\rho_1), \ldots, f(\rho_n)).
  \end{array}
\end{displaymath}
Moreover, if $\OmegaH \ne \emptyset$, then $\OmegaH$
has, at most, topological dimension $n - h$.\\
The following proposition holds.

\begin{prop}
  \label{prop:entropy_restricted}
  Let $\mathcal H\s\{1,\ldots,n+m\}$ be a set of cardinality $h$,
  with $0\le h\le n-1$ and suppose that $\Omega_{\mathcal H}\ne\emptyset$.
  The functional $\mathcal G$, restricted to $\PhiH$, is given by
  \begin{equation}
    \label{eq:entropy_restricted}
    \mathcal G(\rho_1, \ldots, \rho_{n+m}) =
    \sum_{l\in\{1,\ldots,n+m\}\setminus\mathcal H}\left[f(\rho_l)-f(\sigma)\right]
    +\sum_{l\in\mathcal H}\left[f(\sigma)-f(\rho_l)\right].
  \end{equation}
\end{prop}

\begin{proof}
  Fix $(\rho_1\ldots,\rho_{n+m})\in\PhiH$ and $l\in\{1,\ldots,n+m\}$.
  We have some different possibilities.
  \begin{enumerate}
  \item $l\le n$ and $l\in\mathcal H$. In this case the term
    $\sgn(\rho_l-\sigma)\left(f(\rho_l)-f(\sigma)\right)$ becomes
    $f(\sigma)-f(\rho_l)$.

  \item $l\le n$ and $l\not\in\mathcal H$. In this case the term
    $\sgn(\rho_l-\sigma)\left(f(\rho_l)-f(\sigma)\right)$ becomes
    $f(\rho_l)-f(\sigma)$.

  \item $l\ge n+1$ and $l\in\mathcal H$. In this case the term
    $-\sgn(\rho_l-\sigma)\left(f(\rho_l)-f(\sigma)\right)$ becomes
    $f(\sigma)-f(\rho_l)$.

  \item $l\ge n+1$ and $l\not\in\mathcal H$. In this case the term
    $-\sgn(\rho_l-\sigma)\left(f(\rho_l)-f(\sigma)\right)$ becomes
    $f(\rho_l)-f(\sigma)$.
  \end{enumerate}
  Therefore the proof is finished.
\end{proof}

\begin{corollary}
  Let $\mathcal H\s\{1,\ldots,n+m\}$ be a set of cardinality $h$,
  with $0\le h\le n-1$ and suppose that $\Omega_{\mathcal H}\ne\emptyset$.
  The problem of maximizing $\mathcal G$ on the set $\PhiH$
  is equivalent to the problem of maximizing the function $E$,
  defined in~(\ref{eq:E}), on the set $\OmegaH$.  
\end{corollary}

\begin{proof}
  Notice that, by Proposition~\ref{prop:entropy_restricted}, the function
  $\mathcal G$ on the set $\PhiH$ coincides with
  \begin{displaymath}
    \sum_{l\in\{1,\ldots,n+m\} \setminus \mathcal H} f(\rho_l) + C,
  \end{displaymath}
  where $C$ is a constant, depending on $\mathcal H$ and on the initial
  conditions.
  Indeed, if $l\in\mathcal H$, then $\rho_l$
  is completely determined by the initial condition $\rho_{l,0}$.
  More precisely, $\rho_l$ is equal to $\rho_{l,0}$ when $\rho_{l,0}$ is
  a bad datum, while $\rho_l$ is equal to $\sigma$ in the other case.
  Therefore, if $(\rho_1,\ldots,\rho_{n+m})\in\PhiH$, then we deduce that
  \begin{eqnarray*}
    \mathcal G(\rho_1,\ldots,\rho_{n+m}) & = & 
    \sum_{i\in\{1,\ldots,n\}\setminus\mathcal H} f(\rho_i) +
    \sum_{j\in\{n+1,\ldots,n+m\}\setminus\mathcal H} f(\rho_j) + C\\
    & = & \sum_{i\in\{1,\ldots,n\}\setminus\mathcal H} f(\rho_i) +
    \sum_{j\in\{n+1,\ldots,n+m\}} f(\rho_j) + C_1\\
    & = & \sum_{i\in\{1,\ldots,n\}\setminus\mathcal H} f(\rho_i) +
    \sum_{i\in\{1,\ldots,n\}} f(\rho_j) + C_1\\
    & = & 2\sum_{i\in\{1,\ldots,n\}\setminus\mathcal H} f(\rho_i) + C_2,
  \end{eqnarray*}
  where $C_1$ and $C_2$ are constants. Finally note that the function $E$,
  restricted on $\OmegaH$, is given by
  \begin{displaymath}
    E(\gamma_1,\ldots,\gamma_n)=\sum_{i\in\{1,\ldots,n\}\setminus\mathcal H}
    \gamma_i+C_2-C_1.
  \end{displaymath}
  This completes the proof.
\end{proof}

\begin{remark}
  Note that the set $\Phi$ is, in general, disconnected, while the set
  $\Omega$ is convex and so connected.
  The function $\mathcal G$, defined in~(\ref{eq:cal_G}), i.e. the
  entropy function restricted on $\Phi \times \{ \sigma \}$, is
  continuous, since it has not jumps in each connected component of
  $\Phi$.
  Since there is a bijection between the sets $\Omega$ and $\Phi$,
  then we can consider the entropy function on $\Omega$.
  More precisely, define the function
  \begin{displaymath}
    \begin{array}{ccc}
      \Upsilon: \Omega & \longrightarrow & \Phi\\
      (\gamma_1, \ldots, \gamma_{n}) & \longmapsto &
      (\rho_1, \ldots, \rho_{n+m}),
    \end{array}
  \end{displaymath}
  satisfying $f(\rho_i) = \gamma_i$ for every $i \in \{1, \ldots, n\}$,
  and consider the map $\mathcal G \circ \Upsilon: \Omega \to \R$.
  This map, in general, is discontinuous, since
  it can have jumps at every point
  $(\gamma_1, \ldots, \gamma_n) \in \overline{\Omega}_{\mathcal H_1} \cap
  \overline{\Omega}_{\mathcal H_2}$ with $\mathcal H_1 \ne \mathcal H_2$
  different subsets of $\{1,\ldots,n+m\}$ of cardinalities less than
  or equal to $n-1$.
\end{remark}

%
% Riemann solver 2
%
\subsection{Riemann Solver $\mathcal{RS}_2$}\label{ssec:rs2}
In this subsection, we consider the Riemann solver, introduced
in \cite{da-m-p} for data networks; see also \cite{gp-book}. 
The construction can be done in the following way.

  \begin{enumerate}
  \item Fix $\boldsymbol{\theta} \in \Theta$ and define 
    \begin{equation*}
      \Gamma_{inc}=\sum_{i=1}^n\sup\Omega_i,\quad
      \Gamma_{out}=\sum_{j=n+1}^{n+m}\sup\Omega_j,
    \end{equation*}
    then the maximal possible through-flow at the crossing is
    \begin{equation*}
      \Gamma = \min \left\{\Gamma_{inc},\Gamma_{out} \right\}\,.
    \end{equation*}
  
  \item Introduce the closed, convex and not empty sets
    \begin{eqnarray*}
      I & = &
      \left\{ 
        (\gamma_1, \ldots,\gamma_n) \in 
        \prod_{i=1}^n \Omega_i 
        \colon \sum_{i=1}^n \gamma_i = \Gamma 
      \right\}
      \\
      J & = &
      \left\{ 
        (\gamma_{n+1}, \ldots,\gamma_{n+m}) \in 
        \prod_{j=n+1}^{n+m} \Omega_j 
        \colon \sum_{j=n+1}^{n+m} \gamma_j = \Gamma 
      \right\}   \,.
    \end{eqnarray*}

  \item Denote with $(\bar\gamma_1,\ldots,\bar\gamma_n)$ the orthogonal
    projection on the convex set $I$ of the point
    $(\Gamma\theta_1,\ldots,\Gamma\theta_n)$ and with
    $(\bar\gamma_{n+1},\ldots,\bar\gamma_{n+m})$ the
    orthogonal projection
    on the convex set $J$ of the point
    $(\Gamma\theta_{n+1},\ldots,\Gamma\theta_{n+m})$.

  \item For every $i\in\{1,\ldots,n\}$, define $\bar\rho_i$
    either by $\rho_{i,0}$
    if $f(\rho_{i,0})=\bar\gamma_i$, or by the solution to
    $f(\rho)=\bar\gamma_i$
    such that $\bar\rho_i\ge\sigma$.
    For every $j\in\{n+1,\ldots,n+m\}$,
    define $\bar\rho_j$ either by $\rho_{j,0}$
    if $f(\rho_{j,0})=\bar\gamma_j$, or by the solution
    to $f(\rho)=\bar\gamma_j$
    such that $\bar\rho_j\le\sigma$.
    Finally, define $\Rr_2:[0,1]^{n+m}\to[0,1]^{n+m}$ by
  \begin{equation}\label{rs2_rho}
    \Rr_2(\rho_{1,0},\ldots,\rho_{n+m,0})
    =(\bar\rho_1,\ldots,\bar\rho_n,\bar\rho_{n+1},\ldots,\bar\rho_{n+m})\,.
  \end{equation}
\end{enumerate}

The following result holds.

\begin{lemma}
  The function defined in~(\ref{rs2_rho}) satisfies the consistency
  condition
  \begin{equation}\label{eq:stab_rs2}
    \Rr_2(\Rr_2(\rho_{1,0},\ldots,\rho_{n+m,0}))=
    \Rr_2(\rho_{1,0},\ldots,\rho_{n+m,0})
  \end{equation}
  for every $(\rho_{1,0},\ldots,\rho_{n+m,0})\in[0,1]^{n+m}$.
\end{lemma}

For a proof, see~\cite{g-p_generic_J}. We prove now that the Riemann solver
$\Rr_2$ satisfies the entropy condition (E2).

\begin{prop}
  Assume $n = m$ and consider a node $J$ with $n$ incoming
  roads and $m$ outgoing roads.
  The Riemann solver $\Rr_2$ satisfies the entropy condition (E2)
  in the sense of Definition~\ref{def:entropy_RS_E2}.
\end{prop}

\begin{proof}
  Fix an initial condition $(\rho_{1,0}, \ldots, \rho_{n+m,0})$ and define
  $(\bar \rho_1, \ldots, \bar \rho_{n+m}) =
  \Rr_2 (\rho_{1,0}, \ldots, \rho_{n+m,0})$. We have two different cases.
  \begin{description}
  \item[$\Gamma_{inc} \le \Gamma_{out}$.] In this situation, we deduce that
    $\bar \rho_i \le \sigma$ for every $i \in \{1, \ldots, n\}$.
    Thus the entropy reads
    \begin{displaymath}
      \mathcal F (\bar \rho_1, \ldots, \bar \rho_{n+m}, \sigma) =
      n f(\sigma) - \sum_{i=1}^n f(\bar \rho_i)
      -\!\! \sum_{j=n+1}^{n+m}\! \sgn (\bar \rho_j - \sigma)
      \left( f(\bar \rho_j) - f(\sigma) \right).
    \end{displaymath}
    For every $j \in \{n+1, \ldots, n+m\}$, the term
    $-\sgn (\bar \rho_j - \sigma) \left( f(\bar \rho_j) - f(\sigma) \right)$
    can be minorized by $f(\bar \rho_j) - f(\sigma)$ and so
    \begin{eqnarray*}
      \mathcal F (\bar \rho_1, \ldots, \bar \rho_{n+m}, \sigma) & \ge &
      n f(\sigma) - \sum_{i=1}^n f(\bar \rho_i)
      + \sum_{j=n+1}^{n+m} \left( f(\bar \rho_j) - f(\sigma) \right)\\
      & = & (n-m) f(\sigma) = 0.
    \end{eqnarray*}

  \item[$\Gamma_{inc} > \Gamma_{out}$.] In this situation, we deduce that
    $\bar \rho_j \ge \sigma$ for every $j \in \{n+1, \ldots, n+m\}$.
    Thus the entropy reads
    \begin{displaymath}
      \mathcal F (\bar \rho_1, \ldots, \bar \rho_{n+m}, \sigma) = \!
      \sum_{i=1}^{n} \sgn (\bar \rho_i - \sigma)
      \left( f(\bar \rho_i) - f(\sigma) \right)
      + m f(\sigma) - \!\! \sum_{j=n+1}^{n+m} f(\bar \rho_j).
    \end{displaymath}
    For every $i \in \{1, \ldots, n\}$, the term
    $\sgn (\bar \rho_i - \sigma) \left( f(\bar \rho_i) - f(\sigma) \right)$
    can be minorized by $f(\bar \rho_i) - f(\sigma)$ and so
    \begin{eqnarray*}
      \mathcal F (\bar \rho_1, \ldots, \bar \rho_{n+m}, \sigma) & \ge &
      \sum_{i=1}^{n} \left( f(\bar \rho_i) - f(\sigma) \right)
      + m f(\sigma) - \sum_{j=n+1}^{n+m} f(\bar \rho_j)\\
      & = & (m - n) f(\sigma) = 0.
    \end{eqnarray*}
  \end{description}
  The proof is finished.
\end{proof}

In general, the Riemann solver $\Rr_2$ does not satisfy the entropy
condition (E1) even in the case $n=m$, as the next Proposition shows.

\begin{prop}
  The Riemann solver $\Rsol_2$ does not satisfy the entropy condition (E1)
  in the sense of Definition~\ref{def:entropy_RS_E1}.
\end{prop}

\begin{proof}
  Consider a node with $2$ incoming and $2$ outgoing arcs,
  the flux function $f(\rho)=4\rho(1-\rho)$,
  $\boldsymbol\theta=\left(\frac12,\frac12,\frac5{12},\frac7{12}\right)$
  and the equilibrium configuration
  $\left(\frac14,\frac14,\frac12-\frac{\sqrt3}{4\sqrt2},
    \frac12-\frac1{4\sqrt2}\right)$.
  In this case equation~(\ref{eq:entropy_RS_E1}) becomes
  \begin{gather*}
    2\sgn\left(\frac14-k\right)\left(\frac34-f(k)\right)
    -\sgn\left(\frac12-\frac{\sqrt3}{4\sqrt2}-k\right)
    \left(\frac58-f(k)\right)\\-
    \sgn\left(\frac12-\frac1{4\sqrt2}-k\right)
    \left(\frac78-f(k)\right)\ge0
  \end{gather*}
  for every $k\in[0,1]$.
  If $k=\frac14$, then the previous inequality becomes
  \begin{displaymath}
    \left(\frac58-\frac34\right)-\left(\frac78-\frac34\right)\ge0,
  \end{displaymath}
  which is clearly false.
\end{proof}

Indeed, in some special situation, namely for nodes with $2$ incoming
and $2$ outgoing arcs and
$\boldsymbol\theta=\left(\frac12,\frac12,\frac12,\frac12\right)$,
the Riemann solver $\Rr_2$ satisfies the entropy condition (E1).

\begin{prop}
  \label{prop:rs2_E1}
  Fix a node $J$ with two incoming and two outgoing arcs.
  If $\boldsymbol\theta=\left(\frac12,\frac12,\frac12,\frac12\right)$,
  then the Riemann solver $\Rr_2$ satisfies the entropy condition (E1),
  in the sense of Definition~\ref{def:entropy_RS_E1}.
\end{prop}

\begin{proof}
  Consider an equilibrium $(\srho)$ for the Riemann solver $\Rr_2$
  and denote with $g$ the number of good data. We have the following
  possibilities.

  \begin{description}
  \item[$g=4$.] In this case we deduce that
    $(\srho)=\left(\frac12,\frac12,\frac12,\frac12\right)$ and so the
    entropy condition (E1) is satisfied.

  \item[$g=3$.] Consider only the case $\Gamma=\Gamma_{inc}$, since
    the other case $\Gamma=\Gamma_{out}$ is completely symmetric.
    Thus the bad datum is in an incoming arc and so we may assume that
    $\brho_1<\sigma$, $\brho_2\ge\sigma$ and $\brho_3\le\brho_4\le\sigma$.
    Since $\boldsymbol\theta=\left(\frac12,\frac12,\frac12,\frac12\right)$,
    then $\brho_2=\sigma$ and $\brho_3=\brho_4<\sigma$.
    Moreover, the fact that $\f1+\f2=\f3+\f4$ implies that
    \begin{displaymath}
      \brho_1<\brho_3=\brho_4<\brho_2=\sigma.
    \end{displaymath}
    By item~1 of Proposition~\ref{prop:1bad}, the entropy condition (E1) holds.
 
  \item[$g=2$.] Consider only the case $\Gamma=\Gamma_{inc}$, since
    the other case $\Gamma=\Gamma_{out}$ is completely symmetric.
    We have two possibilities: either the bad data are in the incoming arcs
    or one bad datum is in an incoming arc and the other bad datum is in
    an outgoing arc.\\
    Assume first that the bad data are in the incoming arcs.
    Without loss of generality
    we may assume that $\brho_1\le \brho_2<\sigma$ and
    $\brho_3\le\brho_4\le\sigma$.
    Since $\boldsymbol\theta=\left(\frac12,\frac12,\frac12,\frac12\right)$,
    then $\brho_3=\brho_4$ and so,
    the fact that $\f1+\f2=\f3+\f4$ implies that
    \begin{displaymath}
      \brho_1\le\brho_3=\brho_4\le\brho_2<\sigma.
    \end{displaymath}
    By item~1 of Proposition~\ref{prop:2bad}, the entropy condition (E1)
    is satisfied.

    Assume now that one bad datum is in an incoming arc
    and the other bad datum is in an outgoing arc. Without loss of generality
    we may assume that $\brho_1<\sigma<\brho_4$ and
    $\brho_3\le\sigma\le\brho_2$.
    Since $\Gamma=\Gamma_{inc}$, then we deduce that $\brho_2=\sigma$.
    Moreover $\boldsymbol\theta=\left(\frac12,\frac12,\frac12,\frac12\right)$
    implies that $\f3\ge\f4$ and so $\f1\le\f4$, since
    $\f1+\f2=\f3+\f4$. Therefore
    \begin{displaymath}
      \brho_1\le\brho_3\le\brho_2=\sigma<\brho_4\quad\textrm{ and }\quad
      \brho_1<\brho_2.
    \end{displaymath}
    By item~3 of Proposition~\ref{prop:2bad}, the entropy condition (E1)
    is satisfied.

  \item[$g=1$.] Consider only the case $\Gamma=\Gamma_{inc}$, since
    the other case $\Gamma=\Gamma_{out}$ is completely symmetric.
    We have two possibilities: the good datum is in an incoming arc or
    in an outgoing arc.
    Assume first that the good datum is in an incoming arc.
    Without loss of generality, we may consider that
    $\brho_1<\sigma\le\brho_2$ and $\sigma<\brho_3\le \brho_4$.
    Since $\Gamma=\Gamma_{inc}$, then $\brho_2=\sigma$. Moreover
    $\f1+\f2=\f3+\f4$ implies that $\f4\ge \f1$.
    By item~1 of Proposition~\ref{prop:3bad}, the entropy condition (E1)
    is satisfied.

    Assume now that the good datum is in an outgoing arc.
    Without loss of generality, suppose that $\brho_1\le\brho_2<\sigma$,
    $\brho_3\le\sigma<\brho_4$.
    Since $\boldsymbol\theta=\left(\frac12,\frac12,\frac12,\frac12\right)$,
    then $\f3\ge\f4$ and so $\f4\le\f2$ and $\brho_3\ge\brho_1$ since
    $\f1+\f2=\f3+\f4$.
    By item~2 of Proposition~\ref{prop:3bad}, the entropy condition (E1)
    is satisfied.

  \item[$g=0$.] In this case we have that $\Gamma=\Gamma_{inc}=\Gamma_{out}$.
    Without loss of generality, suppose that
    $\brho_1\le\brho_2<\sigma<\brho_3\le\brho_4$ and we conclude by
    Proposition~\ref{prop:4bad}.
  \end{description}

  The proof is finished.
\end{proof}

%
% T node
%
\subsection{Riemann Solver $\Rsol_3$}
In this subsection, we consider the Riemann solver, introduced
in \cite{marigo-piccoli_2008_T_junction} for crossing nodes.
Consider a node $J$ with $n$ incoming and $m=n$ outgoing arcs
and fix a positive coefficient $\Gamma_J$, which is the maximum capacity
of the node.
The construction can be done in the following way.
  \begin{enumerate}
  \item Fix $\boldsymbol{\theta} \in \Theta$.
    For every $i\in\{1,\ldots,n\}$, define 
    \begin{equation*}
      \Gamma_i = \min\left\{\sup\Omega_i,\sup\Omega_{i+n}\right\}.
    \end{equation*}
    Then the  maximal possible through-flow at $J$ is
    \begin{equation*}
      \Gamma = \sum_{i=1}^n \Gamma_i.
    \end{equation*}
  
  \item Introduce the closed, convex and not empty set
    \begin{equation*}
      I=\left\{(\gamma_1,\ldots,\gamma_n) \in \prod_{i=1}^n [0,\Gamma_i] 
        \colon\sum_{i=1}^n\gamma_i=\min\left\{\Gamma,\Gamma_J\right\}\right\}.
    \end{equation*}
    
  \item Denote with $(\bar\gamma_1,\ldots,\bar\gamma_n)$ the orthogonal
    projection on the convex set $I$ of the point
    $(\min \{\Gamma, \Gamma_J\} \theta_1,\ldots,
    \min \{\Gamma, \Gamma_J\} \theta_n)$ and set
    $ (\bar\gamma_{n+1},\ldots,\bar\gamma_{2n})
    =(\bar\gamma_1,\ldots,\bar\gamma_n)$.

  \item For every $i\in\{1,\ldots,n\}$, define $\bar\rho_i$
    either by $\rho_{i,0}$
    if $f(\rho_{i,0})=\bar\gamma_i$, or by the solution to
    $f(\rho)=\bar\gamma_i$
    such that $\bar\rho_i\ge\sigma$.
    For every $j\in\{n+1,\ldots,n+m\}$,
    define $\bar\rho_j$ either by $\rho_{j,0}$
    if $f(\rho_{j,0})=\bar\gamma_j$, or by the solution
    to $f(\rho)=\bar\gamma_j$
    such that $\bar\rho_j\le\sigma$.
    Finally, define $\Rr_3:[0,1]^{n+m}\to[0,1]^{n+m}$ by
  \begin{equation}\label{rs3_rho}
    \Rr_3(\rho_{1,0},\ldots,\rho_{n+m,0})
    =(\bar\rho_1,\ldots,\bar\rho_n,\bar\rho_{n+1},\ldots,\bar\rho_{n+m})\,.
  \end{equation}
\end{enumerate}

The following result holds.

\begin{lemma}
  The function defined in~(\ref{rs3_rho}) satisfies the consistency
  condition
  \begin{equation}\label{eq:stab_rs3}
    \Rr_3(\Rr_3(\rho_{1,0},\ldots,\rho_{n+m,0}))=
    \Rr_3(\rho_{1,0},\ldots,\rho_{n+m,0})
  \end{equation}
  for every $(\rho_{1,0},\ldots,\rho_{n+m,0})\in[0,1]^{n+m}$.
\end{lemma}

For a proof, see Proposition 2.4 of \cite{marigo-piccoli_2008_T_junction}.

\begin{example}
  Consider a node $J$ with $2$ incoming arcs and $2$ outgoing ones,
  $\boldsymbol{\theta} = \left(\frac34, \frac14, \frac34, \frac14 \right)$
  and $\Gamma_J = \frac{64}{75}$.
  Moreover, assume that $f(\rho) = 4 \rho (1-\rho)$.\\
  We easily see that 
  \begin{displaymath}
    (\bar \rho_{1}, \bar \rho_{2}, \bar \rho_{3}, \bar \rho_{4}) =
    \left( \frac15,\, \frac12 + \frac1{10} \sqrt{\frac{59}3},\,
      \frac45,\, \frac12 - \frac1{10} \sqrt{\frac{59}3} \right)
  \end{displaymath}
  is an equilibrium for $\Rr_3$. Thus we have
  \begin{eqnarray*}
    \mathcal F (\bar \rho_{1}, \bar \rho_{2}, \bar \rho_{3}, \bar \rho_{4},
    \sigma ) & = & \left( f(\sigma) - f(\bar \rho_1) \right)
    + \left( f(\bar \rho_2) - f(\sigma) \right)\\
    & & - \left( f(\bar \rho_3) - f(\sigma)\right)
    - \left( f(\sigma) - f(\bar \rho_4) \right)\\
    & = & - \frac{64}{75}.
  \end{eqnarray*}
\end{example}

\begin{example}
  Consider a node $J$ with $2$ incoming arcs and $2$ outgoing ones,
  $\boldsymbol{\theta} = \left(\frac12, \frac12, \frac12, \frac12 \right)$
  and $\Gamma_J = \frac{7}{6}$.
  Moreover, assume that $f(\rho) = 4 \rho (1-\rho)$.\\
  We easily see that 
  \begin{displaymath}
    (\bar \rho_{1}, \bar \rho_{2}, \bar \rho_{3}, \bar \rho_{4}) =
    \left( \frac12 + \frac1{2} \sqrt{\frac{1}2}, \, 
      \frac12 + \frac1{2} \sqrt{\frac{1}3},\,
      \frac12 + \frac1{2} \sqrt{\frac{1}2},\,
      \frac12 - \frac1{2} \sqrt{\frac{1}3} \right)
  \end{displaymath}
  is an equilibrium for $\Rr_3$. Thus we have
  \begin{eqnarray*}
    \mathcal F (\bar \rho_{1}, \bar \rho_{2}, \bar \rho_{3}, \bar \rho_{4},
    \sigma ) & = & \left( f(\bar \rho_1) - f(\sigma) \right)
    + \left( f(\bar \rho_2) - f(\sigma) \right)\\
    & & - \left( f(\bar \rho_3) - f(\sigma)\right)
    - \left( f(\sigma) - f(\bar \rho_4) \right)\\
    & = & 2 \left( f(\bar \rho_2) - f(\sigma) \right) = - \frac{2}{3}.
  \end{eqnarray*}
\end{example}

The following result follows by the previous examples.

\begin{prop}
  The Riemann solver $\Rr_3$ does not satisfy neither the entropy
  condition (E1) nor the entropy condition (E2).
\end{prop}

%
%
%  BIBLIOGRAPHY
%
%

{\small{ \bibliographystyle{abbrv}
    
    \bibliography{1.bib} }}

\end{document}